\documentclass[11pt]{article}

\usepackage{paralist, amsmath, amsthm, amssymb, color}
\usepackage[margin=1in]{geometry}
\usepackage{tikz}
\tikzstyle{bv}=[circle,draw=black!90,fill=black!100,thick,inner
sep=2pt,minimum width=5pt]
\usepackage{hyperref}

\newtheorem*{theorem*}{Theorem}
\newtheorem*{recapprop*}{Proposition \ref{RL:qprop}}
\numberwithin{equation}{section}
\newtheorem{theorem}[equation]{Theorem}
\newtheorem{proposition}[equation]{Proposition}
\newtheorem{lemma}[equation]{Lemma}
\newtheorem{corollary}[equation]{Corollary}

\theoremstyle{definition}

\newtheorem{qst}[equation]{Question}
\newenvironment{question}[1][]{\begin{qst}[#1] \pushQED{\qed}}{\popQED \end{qst}}

\newtheorem{rmk}[equation]{Remark}
\newenvironment{remark}[1][]{\begin{rmk}[#1] \pushQED{\qed}}{\popQED \end{rmk}}

\newcommand{\GL}{\mathbf{GL}}
\DeclareMathOperator{\Sym}{Sym}
\DeclareMathOperator{\Symz}{Sz}
\DeclareMathOperator{\Matz}{{\rm Mat}_0}
\DeclareMathOperator{\inv}{inv}

\DeclareMathOperator{\rk}{rank}
\DeclareMathOperator{\sym}{sym}
\DeclareMathOperator{\symz}{sz}
\DeclareMathOperator{\matz}{{\rm mat}_0}
\DeclareMathOperator{\sq}{sq}
\DeclareMathOperator{\rowspace}{rowspace}
\DeclareMathOperator{\sk}{sk}

\newcommand{\F}{\mathbf{F}}
\DeclareMathOperator{\diag}{diag}

\newcommand{\greta}{g}

\newcommand{\ZZ}{\mathbf{Z}}

\newcommand{\arxiv}[1]{\href{http://arxiv.org/abs/#1}{{\tt arXiv:#1}}}

\title{Matrices with restricted entries and $q$-analogues of permutations}
 
\author{Joel Brewster Lewis \and Ricky Ini Liu \and Alejandro H. Morales \and Greta Panova \and Steven V Sam \and Yan X Zhang} 

\date{February 9, 2012}

\begin{document}

\maketitle

\begin{abstract}
We study the functions that count matrices of given rank over a finite field with specified positions equal to zero.  We show that these matrices are $q$-analogues of permutations with certain restricted values.  We obtain a simple closed formula for the number of invertible matrices with zero diagonal, a $q$-analogue of derangements, and a curious relationship between invertible skew-symmetric matrices and invertible symmetric matrices with zero diagonal.  In addition, we provide recursions to enumerate matrices and symmetric matrices with zero diagonal by rank, and we frame some of our results in the context of Lie theory.  Finally, we provide a brief exposition of polynomiality results for enumeration questions related to those mentioned, and give several open questions.
\end{abstract}


\section{Introduction}\label{sec:intro}

Fix a prime power $q$.  Let $\F_q$ denote the field with $q$ elements and let $\GL(n,q)$ denote the group of $n \times n$
invertible matrices over $\F_q$.  The {\bf support} of a matrix
$(A_{ij})$ is the set of indices $(i,j)$ such that $A_{ij} \ne 0$. Our
work was initially motivated by the following question of Richard Stanley: how many
matrices in $\GL(n,q)$ have support avoiding the diagonal entries? The
answer to this question is
\[
q^{\binom{n-1}{2} - 1}(q-1)^n \left( \sum_{i=0}^n (-1)^i \binom{n}{i}
  [n-i]_q! \right),
\]
which is proven in Proposition~\ref{RL:prop2} as part of a more general result.  This question has a natural combinatorial appeal and is reminiscent of the work of Buckheister \cite{buck} and Bender \cite{bend} enumerating invertible matrices over $\F_q$ with trace zero (see also \cite[Prop. 1.10.15]{ec1}).  It also falls naturally into two broader contexts, the study of $q$-analogues of permutations and the study of polynomiality results for certain counting problems related to algebraic varieties over $\F_q$.

In the former context, we consider the following situation: fix $m,n \geq 1$, $r
\geq 0$, and $S \subset \{(i,j) \mid 1 \le i \le m, 1 \le j \le
n\}$. Let $T_q$ be the set of $m \times n$ matrices $A$ over $\F_q$ 
with rank $r$ and support contained in the complement of
$S$. Also, let $T_1$ be the set of $0$-$1$ matrices with exactly $r$
$1$'s, no two of which lie in the same row or column, and with support
contained in the complement of $S$ (i.e., the set of rook placements avoiding $S$).  We have that $T_q$ is a {\bf $q$-analogue} of $T_1$, in the following precise sense:

\begin{recapprop*} 
  We have $\# T_q \equiv \# T_1 \cdot (q-1)^r \pmod {(q-1)^{r+1}}$.
\end{recapprop*}

In particular, when $\# T_q$ is a polynomial function of $q$ we have that $\# T_q$ is
divisible by $(q-1)^r$ and $\#T_q / (q-1)^r |_{q=1} = \# T_1$.  Thus,
rank $r$ matrices whose support avoids the set $S$ can be seen as a
$q$-analogue of rook placements that avoid $S$. Applying this to our
situation where $S$ is the set of diagonal entries, we get that the set of
invertible matrices avoiding the diagonal is a $q$-analogue of the set of
derangements, a fact that can also be seen directly from the explicit
formula above.  We will give a more conceptual explanation for this in
Section~\ref{sec:bruhat} using the Bruhat decomposition of
$\GL(n,q)$.

Note that for an arbitrary set $S$ of positions, the function $\#T_q$ need not be a polynomial in $q$.  (Stembridge
\cite{stem} gives an example of non-polynomial $\#T_q$ with $n=m=7$, $r=7$, and a set $S$ with $\#S=28$ -- see Figure \ref{fig:Fano}.)  The second context concerns the question of which sets $S$ give a polynomial $\#T_q$ and is deeply related to a speculation of Kontsevich from 1997 (see Stanley \cite{spt} and Stembridge \cite{stem}) that was proven false by Belkale and Brosnan \cite{bb}. We provide further background on this topic in Section~\ref{sec:polynom}.

We close this introduction with a summary of the results of our paper.

Section~\ref{sec:zerodiagonal} is concerned with Stanley's question on
the enumeration of matrices in $\GL(n,q)$ with zero diagonal. We
attack this problem by enumerating larger classes of matrices.  We provide two recursions,
one based on the size of the matrix and the other based on the rank
of the matrix, and we provide a closed-form solution for the first
recursion. 

In Section~\ref{sec:skewsym}, we 
enumerate symmetric matrices in $\GL(n,q)$ whose support avoids
the diagonal in the case that $n$ is even. These matrices may be viewed as a
$q$-analogue of symmetric permutation matrices with zero diagonal,  i.e., fixed point-free involutions.  A curious byproduct of
our formula is that it also counts the number of symmetric matrices in
$\GL(n-1,q)$ as well as the number of skew-symmetric matrices in
$\GL(n,q)$. (This latter equality was obtained earlier by Jones \cite{oj}.) We remark that there is an obvious bijection between
skew-symmetric matrices and symmetric matrices with $0$ diagonal
obtained by reversing some signs, but this map does not preserve the
property of being invertible. In fact, the varieties associated to these three classes of matrices are pairwise non-isomorphic, and we have not found a satisfactory reason that their solution sets have the same size.

In Section~\ref{sec:symmetricrank}, we attack the general
problem of enumerating symmetric matrices with zeroes on the diagonal
with given rank. We provide recursions for arbitrary rank and for full
rank and solve the latter one to obtain the enumeration of symmetric
matrices in $\GL(n,q)$ when $n$ is odd. The situation in this case is
significantly more complicated than in Sections \ref{sec:zerodiagonal} and \ref{sec:skewsym}.

Finally, in Section~\ref{sec:polynom} we prove Proposition \ref{RL:qprop},
discuss the two broader contexts mentioned above and give some open questions
about families of sets $S$ for which $\#T_q(m\times n, S,r)$ is a polynomial in $q$.

While similar methods of proof are used in the various sections, each section is essentially self-contained and so they can be read independently, if desired.

\subsection*{Notation}

Given an integer $n$, we define the $q$-number $[n]_q = \frac{q^n - 1}{q-1}$, 
the $q$-factorial $[n]_q!  = [n]_q \cdot [n-1]_q \cdot [n-2]_q \cdots$
and the $q$-double factorial $[n]_q!! = [n]_q \cdot [n-2]_q \cdot [n-4]_q \cdots$ (with $[0]_q!!=[-1]_q!!=1$).  In addition, we use a number of invented notations; to avoid confusion and for easy reference, we include a table of these functions here.  The last column indicates the sections in which the notation is used.

~

\noindent
\begin{tabular}{llp{3.25in}l} 
set             & $\#$ set      & description & section\\ \hline

$\Matz(n,k,r)$  & $\matz(n,k,r)$ & set of $n\times n$ matrices of rank $r$ over $\F_q$ with first $k$ diagonal entries equal to zero & \ref{sec:Greta's recursion}\\
                & $\matz(r; n, r', *)$ & number of $n \times n$ matrices $A$ of rank $r'$ with lower-right $(n - 1) \times (n - 1)$ block of rank $r$ & \ref{sec:Greta's recursion}\\
                & $\matz(r; n, r', 0)$ & ditto, where in addition we require $A_{1,1} = 0$& \ref{sec:Greta's recursion}\\
$\Sym(n)$       & $\sym(n)$   & set of $n\times n$ symmetric invertible  matrices over $\F_q$ & \ref{sec:skewsym}, \ref{sec:symmetricrank}\\
$\Sym(n, r)$       & $\sym(n, r)$   & set of $n\times n$ symmetric matrices over $\F_q$ of rank $r$ & \ref{sec:symmetricrank} \\
$\Sym_0(n,r)$   & $\sym_0(n,r)$  & set of $n\times n$ symmetric matrices with rank $r$ over $\F_q$ with diagonal entries equal to zero &\ref{sec:symmetricrank} \\
$\Symz(n,k)$ & $\symz(n,k)$ & set of $n\times n$ symmetric invertible matrices with first $k$ diagonal entries equal to zero & \ref{sec:symmetriczero}, \ref{sec:4.2}\\ 
$\Sym_0(n,k,r)$ & $\sym_0(n,k,r)$ & set of $n\times n$ symmetric matrices with rank $r$ with first $k$ diagonal entries equal to zero & \ref{sec:symmetricrank}\\
                & $\sym_0(r; n, r', *)$ & number of $n \times n$ symmetric matrices $A$ of rank $r'$ with lower-right $(n - 1) \times (n - 1)$ block of rank $r$ & \ref{sec:4.1}\\
                & $\sym_0(r; n, r', 0)$ & ditto, where in addition we require $A_{1,1} = 0$& \ref{sec:4.1}\\
$T_q(m\times n,S,r)$ & $\#T_q(m\times n,S,r)$ & set of $m\times n$ matrices over $\F_q$ with rank $r$ and support contained in the complement of $S$ & \ref{sec:intro}, \ref{sec:polynom}\\
$T_1(m\times n, S,r)$ & $\#T_1(m\times n, S,r)$ & set of $0$-$1$ matrices with exactly $r$ $1$'s, no two of which lie in the same row or column, and with support contained in the complement of $S$ & \ref{sec:intro}, \ref{sec:polynom}
\end{tabular}

\subsection*{Acknowledgements} 

We thank Richard Stanley for suggesting the original problem and
Alexander Postnikov for several insights and suggestions. We also
thank David Speyer and John Stembridge for helpful discussions and
comments, and David Blum for technical support. Steven Sam was
supported by an NSF graduate research fellowship and an NDSEG
fellowship. Ricky Liu and Yan X Zhang were supported by NSF graduate research fellowships.

\section{Matrices with zeroes on the
  diagonal} \label{sec:zerodiagonal}

In this section, we consider the problem of counting invertible matrices over
$\F_q$ with zero diagonal. In Section~\ref{rickyrec} we recursively count full rank matrices of rectangular shape with all-zero diagonal and in Section~\ref{sec:Greta's recursion} we recursively count square matrices by rank and number of zeroes on the diagonal. We solve the first
recursion and obtain a closed form formula for the number of
invertible matrices with zeroes on the diagonal.  These numbers give an enumerative $q$-analogue of the derangements, i.e., dividing all factors of $q - 1$ and setting $q = 1$ in the result gives the number of derangements, and we give a conceptual proof of this fact in Section~\ref{sec:bruhat}.

\subsection{Recursion by size} \label{rickyrec}

For $1 \leq k \leq n$, denote by $f_{k,n}$ the number of $k \times n$ matrices $A$ over $\F_q$ such that $A$ has full rank $k$ and such that $A_{ii}=0$ for $1\leq i \leq k$. We first show that $f_{k,n}$ satisfies a simple
recurrence, and then we solve this recurrence for an explicit formula
for $f_{k,n}$. In particular, we will have a formula for $f_{n,n}$,
the number of invertible $n \times n$ matrices with zeroes on
the diagonal.

\begin{proposition} \label{RL:prop1}
For $1 \leq k < n$, the number $f_{k, n}$ of $k\times n$ matrices with rank $k$ and diagonal entries equal to $0$ satisfies the recursion 
\[
f_{k+1,n} = q^{k-1}(q-1)(f_{k, n} \cdot [n-k]_q  -  f_{k,n-1})
\]
with initial values $f_{1,n} = q^{n-1}-1$.
\end{proposition}

\begin{proof}
  Let $A$ be a $k \times n$ matrix of rank $k$ and such that $A_{ii}=0$ for $1\leq i\leq k$,
  let $V \subset \F_q^n$ be the row span of $A$, and let $W
  \subset \F_q^n$ be the subspace of vectors with $(k+1)$-st
  coordinate $0$. To extend $A$ to a $(k+1)\times n$ full rank matrix with zero diagonal,
  we choose for the final row any vector in $W \backslash V$. We have
  two cases:
\begin{description}
\item[Case 1:] the $(k+1)$-st column of $A$ is not entirely $0$.  In this case, $V \not \subset W$.  Since $W$ is the kernel of the linear form $x_{k + 1}$, $V \cap W$ is the kernel of
  the linear form $x_{k+1}|_V$.  This form is not identically zero, so $\dim (V \cap W) = \dim V - 1 = k-1$. 
  Thus, there are $\#(W \backslash V) = q^{n-1}-q^{k-1}$ choices for
  the final row.
\item[Case 2:] the $(k+1)$-st column of $A$ is entirely $0$.  In this case, $V \subset
  W$ and so $\dim (V \cap W) = \dim V = k$.  Thus, there are $\#(W \backslash V) = q^{n-1}-q^k$ choices for the final row.
\end{description}

The total number of matrices $A$ is $f_{k,n}$, while the number with
$(k+1)$-st column entirely 0 is $f_{k,n-1}$ (as we can remove the
column to get a full rank $k \times (n-1)$ matrix). It follows that
\begin{align*}
  f_{k+1,n} & = (f_{k,n} - f_{k,n-1})(q^{n-1}-q^{k-1}) + f_{k,n-1}(q^{n-1}-q^k) \\
  & = f_{k,n}(q^{n-1}-q^{k-1})-f_{k,n-1}(q^{k}-q^{k-1}) \\
  & = q^{k-1}(q-1)(f_{k,n}\cdot[n-k]_q-f_{k,n-1}),
\end{align*}
as desired.
\end{proof}

With this observation, we can calculate the number $f_{k,n}$ explicitly.
\begin{proposition} \label{RL:prop2}
For $1 \leq k \leq n$, 
\[
f_{k,n} = q^{\binom{k-1}{2}}(q-1)^k \left( q^{-1} \sum_{i=0}^k
  (-1)^i\binom{k}{i}\frac{[n-i]_q!}{[n-k]_q!}\right)
\]
is the number of $k \times n$ matrices of rank $k$ with zeroes on
the diagonal. In particular, 
\[
f_{n,n} = q^{\binom{n-1}{2}}(q-1)^n \left( q^{-1} \sum_{i=0}^n
  (-1)^i\binom{n}{i}[n-i]_q!\right)
\]
is the number of invertible $n \times n$ matrices with zeroes on
the diagonal.
\end{proposition}

\begin{proof}
We proceed by induction on $k$. When $k=1$, we need to show that 
\[
f_{1,n}=(q-1)q^{-1}([n]_q - 1)=(q-1)[n-1]_q=q^{n-1}-1.
\]
This is clear since $f_{1,n}$ counts nonzero elements of $\F_q^n$ with first coordinate $0$.

For the inductive step, assuming the claim holds for $k$ we have by Proposition \ref{RL:prop1} that 
\begin{align*}
  f_{k+1,n}&=q^{k-1}(q-1)(f_{k,n}\cdot[n-k]_q-f_{k,n-1})\\ 
  &= q^{\binom{k}{2}}(q-1)^{k+1}\cdot q^{-1}\left( \sum_{i=0}^k
    (-1)^i\binom{k}{i}\frac{[n-i]_q!}{[n-k]_q!}\cdot
    [n-k]_q-\sum_{i=0}^k
    (-1)^i\binom{k}{i}\frac{[n-i-1]_q!}{[n-k-1]_q!}\right)\\  
  &=q^{\binom{k}{2}}(q-1)^{k+1}\cdot q^{-1}\left( \sum_{i=0}^k
    (-1)^i\binom{k}{i}\frac{[n-i]_q!}{[n-k-1]_q!}+\sum_{i=1}^{k+1}
    (-1)^i\binom{k}{i-1}\frac{[n-i]_q!}{[n-k-1]_q!}\right)\\ 
  &=q^{\binom{k}{2}}(q-1)^{k+1}\left( q^{-1}\sum_{i=0}^{k+1}
    (-1)^i\binom{k+1}{i}\frac{[n-i]_q!}{[n-k-1]_q!}\right)
\end{align*}
as desired.
\end{proof}

\begin{remark} \label{remark:qderangements} In the expression for
  $f_{n,n}$, the $q=1$ specialization of the alternating sum is
  \[
  \sum_{i=0}^n (-1)^i\binom{n}{i}(n-i)! = n!\sum_{i=0}^n
  \frac{(-1)^i}{i!},
  \]
  which is the $n$-th derangement number $d_n$. The above proof does
  not ``explain'' this fact, so we provide another proof that better
  elucidates the result.
\end{remark}

\subsection{The Bruhat decomposition} \label{sec:bruhat}

The square matrices with all-zero diagonal of full rank are exactly the matrices in
$\GL(n,q)$ with all-zero diagonal.  This slight rephrasing motivates us to recall the Bruhat decomposition of $\GL(n,q)$: let $B$
be the subgroup of $\GL(n, q)$ of lower triangular matrices.  We have a double
coset decomposition
\[
\GL(n,q) = \coprod_{w \in \mathfrak{S}_n} B w B,
\]
into Bruhat cells, where we abuse notation by using $w \in \mathfrak{S}_n$ to also denote
the corresponding permutation matrix. Hence, every matrix $g \in
\GL(n,q)$ can be uniquely written in the form $bg^w$ where $b \in B$,
$w \in \mathfrak{S}_n$, and $g^w$ is a matrix with a 1 in location
$(w(i),i)$ for all $i$, and such that $g^w_{w(i),j} = 0$ whenever $j >
i$ and $g^w_{k,i} = 0$ whenever $k > w(i)$. For a fixed $w$, call this
set of matrices $B(n,w)$. We will enumerate the number of elements in
each of these sets with zeroes on the diagonal.

First we need a lemma. Let $z(N,q)$ be the number of solutions to the equation $\sum_{j=1}^N A_i B_i = 0$ where $A_i, B_i \in \F_q$, and let $y(N,q)$ be the number of solutions to the equation $\sum_{j=1}^N A_i B_i = \alpha$ where $\alpha$ is some nonzero element of $\F_q$. This number is independent of which $\alpha$ we choose. 

\begin{lemma}\label{Bru:lem1}
We have $ z(N,q) \equiv 1 \pmod{(q-1)}$ and $ y(N,q) \equiv 0 \pmod{(q-1)} $.
\end{lemma}
\begin{proof}
  In both cases, the multiplicative group $\F_q^\times$ acts
  on the set of solutions: for any $x \in \F_q^\times$ and a
  set of solutions $\{A_i\} \cup \{B_i\}$, we can multiply all the
  $A_i$ by $x$ and all the $B_i$ by $x^{-1}$ to get another
  solution. Each orbit has size $q-1$ except for one, namely when all
  the $A_i$ and $B_i$ are zero. The result follows.
\end{proof}

\begin{remark} It isn't too hard to prove the formulas 
  \[
  z(n,q) = q^{n-1} (q^n + q - 1), \quad y(n,q) = q^{n-1} (q^n - 1),
  \]
  but we will not need these exact counts here.
\end{remark}

Now we are ready to give an alternative explanation for
Remark~\ref{remark:qderangements}.
\begin{theorem}
Consider the number of matrices in $B(n,w)$ with zero diagonal. In the case that $w$ is a derangement, this number will be of the form $(q-1)^n(q^i + (q-1)f(q))$ for some polynomial $f$ and some $i$. In all other cases, this number will be divisible by $(q-1)^{n+1}$. 
\end{theorem}

\begin{proof}
  Write $g = bg^w$. Then the diagonal terms of $g$ are $g_{d,d} =
  \sum_{i=1}^n b_{d,i} g^w_{i,d}.$ For each $d=1,\dots,n$, the
  variables appearing in the sum are distinct, so we can count the
  solutions to each equation $g_{d,d} = 0$ independently of one
  another. We need to consider three cases, corresponding to $w(d) >
  d$, $w(d) = d$, and $w(d) < d$. Let $N_w^d$ be the number of terms
  in the $d$-th column of $g^w$ that are strictly above the diagonal
  and not forced to be 0 by our definition of $g^w$.

\begin{compactenum}[1.]
\item Suppose that $w(d) > d$. Then the equation $g_{d,d} = 0$ becomes 
\[
-b_{d,d}g^w_{d,d} = \sum_{i=1}^{d-1} b_{d,i} g^w_{i,d}.
\] 
We know that $b_{d,d} \ne 0$ and $g^w_{d,d}$ is either 0 and could be nonzero if $w(j) \ne d$ for $j < i$. If $g^w_{d,d} = 0$, then we are counting solutions to the equation 
\[
\sum_{j=1}^{N_w^d} A_j B_j = 0
\]
where $A_j, B_j \in \F_q$ are arbitrary. If $g^w_{d,d} \ne 0$, we are counting solutions to the equation $\sum_{j=1}^{N_w^d} A_j B_j = \alpha$ where $\alpha$ is some nonzero element of $\F_q$. So we conclude that the contribution is either $(q-1)(z(N_w^d, q) + y(N_w^d, q))$ or $(q-1)z(N_w^d, q)$. 
\item Suppose that $w(d) < d$. Then the equation $g_{d,d} = 0$ becomes 
\[
-b_{d,w(d)} = \sum_{i=1}^{w(d)-1} b_{d,i} g^w_{i,d}.
\] 
We know that $b_{d,w(d)}$ can be arbitrary. So the number of solutions
is of the form $z(N_w^d, q) + (q-1)y(N_w^d, q)$. The variables
$b_{d,j}$ for $w(d) < j \le d$ have not been involved, so we can
choose them arbitrarily subject to $b_{d,d} \ne 0$. Hence we get a
contribution of
\[
(q-1)q^{d-1-w(d)}(z(N_w^d, q) + (q-1)y(N_w^d, q)).
\] 
\item Finally, suppose that $w(d) = d$. Then the equation $g_{d,d} = 0$ becomes 
\[
-b_{d,d} = \sum_{i=1}^{d-1} b_{d,i} g^w_{i,d},
\] 
and $b_{d,d} \ne 0$. Hence the contribution is $(q-1)y(N_w^d, q)$. 
\end{compactenum}

Since the count for each $d$ is independent of the other $d$, we
multiply the above contributions to get an expression for the number of
matrices in $B(n,w)$ with zero diagonal. Using Lemma \ref{Bru:lem1}, we get the desired
conclusion.
\end{proof}

Hence, when we specialize $f_{n,n}$ to $q=1$ after dividing by
$(q-1)^n$, we will get a contribution of $1$ for each derangement and
$0$ for all other permutations, which gives us the derangement
numbers.

\subsection{Recursion by rank}\label{sec:Greta's recursion}

In this section, we use recursive methods to attack the problem of
enumerating square matrices with a prescribed number of zeroes on the
diagonal by rank.  We use the following strategy: each $n \times n$
matrix can be inflated to $q^{2n + 1}$ different $(n + 1) \times (n +
1)$ matrices, and we count these by keeping careful track of what
their rank is and how many zeroes they have on the
diagonal.  Unfortunately, in this case we are unable to solve the 
recursion to provide a closed formula for the result.

Let $\matz(n,k,r)$ be the number of $n\times n$ matrices over $\F_q$ 
of rank $r$ whose first $k$ diagonal entries are zero (and the other 
diagonal entries may or may not be zero).

\begin{proposition} \label{prop:Greta's recursion with k zeroes} We
  have the following recursion:
  \[
  \matz(n+1,k+1,r+1) = \frac{1}{q}\matz(n+1,k,r+1) +
  (q^{r+1}-q^r)\matz(n,k,r+1) - (q^r-q^{r-1})\matz(n,k,r)
  \]
  with initial conditions 
  \[
  \matz(n,0,r) = \frac{q^{\binom{r}{2}}(q-1)^r}{[r]_q!}
  \left(\prod_{i=0}^{r-1} [n-i]_q\right)^2.
  \]
\end{proposition} 

\begin{proof} 
  Let $B$ be any (fixed) $n \times n$ matrix of rank $r$ whose first
  $k$ diagonal entries are zero.  There are $q^{2n + 1}$ $(n + 1)
  \times (n + 1)$ matrices of the form $A = \begin{bmatrix} a & u \\ v
    & B \end{bmatrix}$ where $a$ is an element of $\F_q$, $u$ is a row
  vector over $\F_q$ of length $n$, and $v$ is a column vector over
  $\F_q$ of length $n$.  We seek to enumerate these matrices by rank,
  taking into account whether or not $a = 0$.  Afterwards, we will sum
  over all matrices $B$ to arrive at the desired recursion.  Observe
  that the number of matrices $A$ of rank $r'$ summed over all such
  $B$, $a$, $u$ and $v$ is $\matz(n + 1, k, r')$, while the number of
  these matrices where in addition we require $a = 0$ is $\matz(n + 1,
  k + 1, r')$.

  As a simplifying step, we show that the number of such $A$ for a
  fixed $B$ depends only on $n$ and $r$. Since $B$ is of rank $r$,
  there are some matrices $X, Y \in \GL(n,q)$ such that $XBY =
  \diag(1^r, 0^{n - r})$.
  It follows that for any matrix $A$ of the form
\[
A = \begin{bmatrix} a & u \\ v & B \end{bmatrix}
\]
we have for $D = \diag(1^r, 0^{n - r})$ that
\[
\begin{bmatrix} 1 & 0 \\ 0 & X \end{bmatrix} \cdot A \cdot \begin{bmatrix} 1 & 0 \\ 0 & Y \end{bmatrix} = \begin{bmatrix} a & uY \\ Xv & D \end{bmatrix},
\]
where by an abuse of notation we use the symbol $0$ to represent both
the all-zero column and row vector of length $n$.  Since $X$ and $Y$
are invertible and we are interested in what happens as $u$ and $v$
vary over all possible $n$-vectors, this computation reduces our
problem to considering the case that $B = D = \diag(1^r, 0^{n - r})$,
regardless of the value of $k$. Thus, we define $\matz(r; n + 1, r',
*)$ to be the number of $(n + 1) \times (n + 1)$ matrices $A$ of rank
$r'$ associated to the matrix $B$ of rank $r$, and we define $\matz(r;
n + 1, r', 0)$ to be the number of such matrices where in addition $a
= 0$.

Observe that when we ``glue'' an extra row or column to a matrix, we increase its rank by either $1$ or $0$.  Since every matrix of the form that interests us arises by gluing one row and one column to $B$, we have that the rank of $A$ is one of $r$, $r + 1$ and $r + 2$.  We now separately consider these three cases.

\begin{description}
\item[Case 1:] $\rk(A)=r$.  In order to have $\rk A = r$, all the entries of $u$ and $v$ after the first $r$ must be equal to zero, i.e., $u=(u_1,\ldots,u_r,0,\ldots,0)$ and $v=(v_1,\ldots,v_r,0,\ldots,0)^T$.  In addition, if we apply Gaussian elimination and use the entries of $B$ to eliminate $u$ and $v$, the $(1, 1)$-entry in the result is $a - uBv = a - \sum_{i=1}^r u_iv_i$, and this entry must be equal to $0$.  Conversely, whenever $a$, $u$ and $v$ satisfy these conditions we have $\rk A = r$.  Thus we have in this case the following conclusions:
\begin{compactenum}[(i)]
\item If we do not restrict $a$ to be zero, the total number of matrices $A$ is $q^{2r}$.  In other words, $\matz(r; n + 1, r, *) = q^{2r}$.

\item Under the additional restriction $a = 0$, either $u=0$ and $v$
  is arbitrary or $u\neq 0$ and $v$ is orthogonal to $u$, so the total
  number of matrices $A$ is $q^r + (q^r - 1)q^{r - 1} = q^{2r - 1} +
  q^r - q^{r - 1}$.  In other words, $\matz(r; n + 1, r, 0) = q^{2r -
    1} + q^r - q^{r - 1}$.
\end{compactenum}

The above two pieces of information imply the equation
\begin{equation}\label{gr:rec1}
\matz(r; n+1,r,0) = \frac{1}{q} \matz(r; n+1,r,*) + q^r - q^{r-1}.
\end{equation}

\item[Case 2:] $\rk(A)=r+2$. In order to have $\rk(A) = r + 2$, both
  $u$ and $v$ must have a nonzero entry among their last $n-r$
  entries.  This is also clearly sufficient, regardless of the value
  of $a$.  Thus, we have in this case the following conclusions:
  \begin{compactenum}[(i)]
  \item If we do not restrict $a$ to be zero, the total number of
    matrices $A$ is $q \cdot q^{2r} \cdot (q^{n - r} - 1)^2 = q(q^n -
    q^r)^2$.  In other words, $\matz(r; n + 1, r + 2, *) = q(q^n -
    q^r)^2$.
  \item Under the additional restriction $a = 0$, the number of
    matrices $A$ is $(q^n - q^r)^2$.  In other words, $\matz(r; n + 1,
    r + 2, 0) = (q^n - q^r)^2$.
  \end{compactenum}

  The above two pieces of information imply the equation
  \begin{equation} \label{gr:rec2} \matz(r; n+1,r+2,0) = \frac{1}{q}
    \matz(r; n+1,r+2,*).
  \end{equation}

\item[Case 3:] $\rk(A) = r + 1$.  The case $\rk(A) = r + 1$ consists
  of all possible choices of $a$, $u$ and $v$ that do not fall into
  either of the preceding cases.  Thus, we have the following
  conclusions:
  \begin{compactenum}[(i)]
  \item If we do not restrict $a$ to be zero, the total number of
    matrices $A$ is 
    \[
    q^{2n + 1} - q^{2r} - q(q^n - q^r)^2 = 2q^{n+r+1} - q^{2r+1} -
    q^{2r}.
    \]
    In other words, $\matz(r; n + 1, r + 1, *) = 2q^{n+r+1} - q^{2r+1}
    - q^{2r}$.
  \item Under the additional restriction $a = 0$, the number of
    matrices $A$ is 
    \[
    q^{2n} - (q^{2r - 1} + q^r - q^{r - 1}) - (q^n - q^r)^2 = 2q^{n+r}
    - q^{2r} - q^{2r-1} - q^r +q^{r-1}.
    \]
    In other words, $\matz(r; n + 1, r + 1, 0) = 2q^{n+r} - q^{2r} -
    q^{2r-1} - q^r +q^{r-1}$.
\end{compactenum}

The above two pieces of information imply the equation
\begin{equation} \label{gr:rec3} \matz(r; n+1,r+1,0) = \frac{1}{q}
  \matz(r; n+1,r+1,*)-q^r + q^{r-1}.
\end{equation}

\end{description}

We now change our perspective and consider the set of all $(n + 1)
\times (n + 1)$ matrices $A$ of rank $r + 1$ whose first $k + 1$
diagonal entries are equal to $0$.  Parametrizing these matrices by
the $n \times n$ submatrix $B$ that results from removing the first
row and first column, we have
\begin{equation} \label{greta:pfreceqA}
\matz(n + 1, k + 1, r + 1) = \sum_{B} \matz(\rk(B); n + 1, r + 1, 0)
\end{equation}
where the sum is over all $n \times n$ matrices $B$ whose first $k$ diagonal entries are zero.  The summands on the right are zero unless $\rk(B) \in \{r - 1, r, r + 1\}$, and so splitting the right-hand side according to the rank of $B$ gives
\begin{multline} \label{greta:pfreceqB}
\matz(n + 1, k + 1, r + 1) = \matz(n, k, r - 1) \cdot \matz(r - 1; n + 1, r + 1, 0) + \\ 
+ \matz(n, k, r + 1) \cdot \matz(r + 1; n + 1, r + 1, 0) + \matz(n, k, r) \cdot \matz(r; n + 1, r + 1, 0).
\end{multline}
Now we may substitute from Equations \eqref{gr:rec1}, \eqref{gr:rec2}, and \eqref{gr:rec3} and collect the terms with coefficient $\frac{1}{q}$ to conclude that 
\[
\matz(n+1,k+1,r+1) = \frac{1}{q}\matz(n+1,k,r+1) + (q^{r+1}-q^r)\matz(n,k,r+1) - (q^r-q^{r-1})\matz(n,k,r),
\]
as desired.  

The initial values for this recursion are the numbers $\matz(n, 0, r)$ of $n \times n$ matrices of rank $r$ with no prescribed values.  A simple counting argument (see, for example, \cite[Section 1.7]{mor}) gives that the number of these is $\displaystyle \matz(n,0,r) = \frac{q^{\binom{r}{2}}(q-1)^r}{[r]_q!}\left(\prod_{i = 0}^{r - 1}[n - i]_{q}\right)^2$.
\end{proof}

The preceding recursion works by reducing the number of zeroes required to lie on the diagonal.  However, we can easily modify the proof to work only with matrices of all-zero diagonal. 

\begin{corollary} \label{prop:Greta's recursion} For $r \ge 0$, the number $\greta(n,r)$ of $n \times n$ matrices over $\F_q$ of rank $r$ and with zero diagonal satisfies the recursion
  \begin{multline*}
    \greta(n+1,r+1) = (q^n - q^{r-1})^2 \greta(n, r - 1)  + (q^{2r+1} + q^{r+1} - q^r) \greta(n, r + 1)  \\
     + (2q^{n+r} - q^{2r} - q^{2r-1} - q^r + q^{r-1}) \greta(n,r) 
  \end{multline*}
with initial conditions $\greta(n,0) = 1$, $\greta(n, -1) = 0$ and $\greta(1,1) = 0$.
\end{corollary}

\begin{proof} 
  We use the same setup as in Proposition~\ref{prop:Greta's recursion
    with k zeroes}. That is, we let $B$ be any (fixed) $n \times n$
  matrix of rank $r$ whose all diagonal entries are zero.  There are
  $q^{2n}$ $(n + 1) \times (n + 1)$ matrices of the form $A
  = \begin{bmatrix} 0 & u \\ v & B \end{bmatrix}$ where $u$ is a row
  vector over $\F_q$ of length $n$, and $v$ is a column vector over
  $\F_q$ of length $n$.  We seek to enumerate these matrices by
  rank. Observe that the number of matrices $A$ of rank $r'$ summed
  over all such $B$, $u$ and $v$ is $\greta(n,r')$.

  By the same simplifying step as in the proof of
  Proposition~\ref{prop:Greta's recursion with k zeroes}, it is enough to count matrices of the form $\begin{bmatrix} 0
    & u \\ v & D \end{bmatrix}$ where case where $D = \diag(1^r, 0^{n - r})$.  We re-use the notation $\matz(r; n + 1, r', 0)$ from the proof
  of Proposition~\ref{prop:Greta's recursion with k zeroes} for the
  number of such matrices.
  
  By Equation \eqref{greta:pfreceqB} we have for $r \ge 1$ the recursion
\begin{multline*}
  \greta(n+1,r+1)= \greta(n, r - 1) \cdot \matz(r - 1; n + 1, r + 1,
  0) \\+ \greta(n, r + 1) \cdot \matz(r + 1; n + 1, r + 1, 0)+
  \greta(n,r) \cdot \matz(r; n + 1, r + 1, 0).
\end{multline*}
 When $r = 0$, the same recursion holds if we interpret $\greta(n, -1) = 0$.

 We obtain the desired result by substituting the formulas for
 $\matz( \,\cdot\, ; \,\cdot\, , \,\cdot\, , 0)$ from the conclusions (ii) of the three
 cases in the proof of Proposition \ref{prop:Greta's recursion with k zeroes}.
\end{proof}

\begin{corollary}
For $n \geq r \geq 0$, the number $g(n, r) = \matz(n, n, r)$ of $n \times n$ matrices over $\F_q$ of rank $r$ with zero diagonal is
\[
g(n, r) = (q - 1)^r \sum_{k = 0}^{n - r} \sum_{i = 0}^n (-1)^{k + r + n + i} q^{\binom{n}{2} + \binom{k}{2} - nk - r} \binom{n}{i} \frac{[n + k - i]_q!}{([k]_q!)^2 \cdot [n - r - k]_q!}.
\]
\end{corollary}

\begin{proof}
Induct, taking as base cases $r = n + 1$ and $r = n$.  (The former is
trivial and the latter easily reduces to Proposition \ref{RL:prop2}.) Then check that the formula above satisfies the recursion of Corollary \ref{prop:Greta's recursion}.
\end{proof}

\section{Symmetric and skew-symmetric matrices} \label{sec:skewsym}

A natural next step is to consider symmetric matrices, which are (at least morally) a $q$-analogue of involutions, suggesting the possibility of interesting combinatorial results. This also brings us closer to a speculation by Kontsevich (see Section~\ref{sec:polynom}). In this section, we begin by enumerating symmetric invertible matrices over $\F_q$ whose diagonal is all zero, a $q$-analogue of involutions with no fixed points.  This leads to two very unintuitive facts: in Section~\ref{sec:symmetriczero}, we show that the number of these matrices of size $2n$ is the same as the number of invertible symmetric matrices of size $2n-1$;  in Section~\ref{sec:curious} and Section~\ref{sec:curious2}, we show in two different ways that both of these numbers equal the number of invertible skew-symmetric matrices of the size $2n$.

While extending the approach of Section \ref{rickyrec} to the case of symmetric matrices seems impossible, the ideas of Section \ref{sec:Greta's recursion} can be adjusted to work in this context.  The major complicating factor is that the bilinear form $u B v$ that we worked with implicitly in Section \ref{sec:Greta's recursion} must be replaced with the quadratic form $v B v^T$.  Quadratic forms behave very differently in even and odd characteristic, so we give the following \emph{proviso}:

\begin{remark}
The proofs in this section are only valid for $q$ odd, unless otherwise noted.
\end{remark}

Of course, some results still hold when $q$ is even: for example,
Proposition~\ref{prop:curious} is equivalent to
Theorem~\ref{thm:clover} for the silly reason that in even
characteristic, the skew-symmetric invertible matrices are exactly the
symmetric invertible matrices with zeroes on the diagonal. For a more
thorough treatment of the case $q$ even, see \cite{macwilliams} and
\cite{spt}.

\subsection{Symmetric matrices with zeroes on the
  diagonal} \label{sec:symmetriczero} 

Let $\Sym_0(n)$ denote the set of $n \times n$ symmetric matrices in $\GL(n, q)$ with zero diagonal and let $\sym_0(n) = \# \Sym_0(n)$.  Similarly, let $\Symz(n, k)$ be the set of $n \times n$ symmetric matrices in $\GL(n, q)$ whose first $k$ diagonal entries are zero and let $\symz(n, k) = \# \Sym_0(n, k)$, so $\Symz(n, n) = \Sym_0(n)$.

In \cite[Theorem 2]{macwilliams}, MacWilliams shows

\begin{theorem*} The number of symmetric invertible matrices (for any characteristic) is
  \begin{equation} \label{eqn:invsymmetric} 
 \sym(n) = q^{\binom{n+1}{2}} \prod_{j=1}^{\lceil n/2 \rceil} (1-q^{1-2j}).
\end{equation}
\end{theorem*}

We now show that when $n$ is even, $\sym(n-1)$ also counts $n\times n$ symmetric invertible matrices with zero diagonal. Observe that this implies that $\sym_0(n)$ is an enumerative $q$-analogue of $(n-1)!!$, the number of fixed point-free involutions in $\mathfrak{S}_n$.

\begin{theorem} \label{thm:clover} When $n$ is even, the number of $(n
  - 1) \times (n - 1)$ symmetric invertible matrices is equal to the
  number of $n \times n$ symmetric invertible matrices with zero
  diagonal, i.e., $\sym(n - 1) = \sym_0(n)$.
\end{theorem}

Note that the case for $q$ even was done by MacWilliams \cite[Theorems 2, 3]{macwilliams} (see also \ref{mac:qeveneq2}). To prove this for $q$ odd, we will use a lemma.  Trivially, a $q^{-k}$-fraction of
all matrices have first $k$ diagonal entries equal to zero.  Naively,
one might expect this to carry over approximately to other classes of
matrices, e.g., one might guess that about $q^{-k}$ of all matrices in
$\Sym(n) = \Symz(n, 0)$ belong to $\Symz(n, k)$.  Our lemma shows
that, remarkably, this estimate is actually exact when $n$ is even.

\begin{lemma} \label{lemma:numberzeroes} 
 When $n$ is even, $q^k \symz(n, k) = \symz(n, 0)$.
\end{lemma}

\begin{proof} It is enough to show that $\symz(n, k) = q \cdot  \symz(n, k +
  1)$. To do this, we will partition both $\Symz(n,k)$ and
  $\Symz(n,k+1)$ into two disjoint sets and construct a $q$-to-1
  mapping from each piece of $\Symz(n,k)$ to a piece of
  $\Symz(n,k+1)$. The two pieces will be defined by whether or not the
  bottom-right $(n-1) \times (n-1)$ submatrix is invertible.

  Recall that $\Symz(n,k)$ is the set of invertible symmetric $n
  \times n$ matrices whose first $k$ diagonal entries are required to
  be 0. If we only care about the cardinality of this set, then the
  actual position of the forced zero entries on the diagonal is
  irrelevant because we can permute rows and columns, so for
  convenience let's temporarily redefine $\symz(n, k)$ to be the
  number of matrices whose diagonal pattern is $(\alpha, a_1, \ldots,
  a_{n - k-1}, 0, \ldots, 0)$, where the $a_i$ are ``free'' entries
  with $k$ zeroes trailing them. So write a generic matrix in $\Symz(n,
  k)$ as $A = \begin{bmatrix} \alpha & v \\ v^T & B \end{bmatrix}$
  where $B$ is an $(n-1) \times (n-1)$ matrix whose last $k$ diagonal entries are zero.

  If $\det B = 0$ then $\det A$ does not depend on $\alpha$.  Thus, changing $\alpha$ to $0$
  results in a new matrix with the same determinant as $A$; in particular, the resulting matrix is also invertible.  
  This operation gives a $q$-to-$1$ map from $\Symz(n, k)$ to $\Symz(n, k+1)$ in the case that $B$ is not invertible.

  Otherwise, when $B$ is invertible, we want to count matrices in
  $\Symz(n, k)$ having $B$ in the bottom-right corner. We can build
  such matrices as follows. For every choice of $v$, there is a unique
  $(n-1)$-tuple $c$ such that $c^T B = v$. Then we can choose $\alpha$
  to be anything other than $c^T \cdot v^T = c^T B^T c = c^T B c$ to
  get a nonsingular matrix.

  We want to show the following: the number of matrices in
  $\Symz(n, k)$ with $B$ as the lower right corner is $q$ times the number of matrices in $\Symz(n, k)$ with $B$ in
  the lower right corner  and $\alpha
  = 0$. The latter number is $q^{n-1}(q-1)$ and the
  former number is $q^{n-1} - N$ where $N$ is the number of choices
  for $c$ such that $c^T B c=0$. Thus, what we want to show is
  equivalent to the statement that $c^T B c = 0$ for $1/q$ of the total number
  of choices for $c$, or that $c^T B c = 0$ for $q^{n-2}$ values of $c$.

  To see this, we proceed as follows. Let $C = (\det B) B$. Clearly
  $c^T C c = 0$ if and only if $c^T B c = 0$, so we can work with $C$ in place of $B$.  We have
  $\det C = (\det B)^n$, and since $n$ is even this is a square in $\F_q$. Using the
  classification of symmetric bilinear forms over fields of odd
  characteristic \cite[Theorem 1.22]{wan}, we can write $C = MM^T$ for
  some nonsingular matrix $M$. So setting $d=M^Tc$, we only need to count the number of
  $d$ such that $d^T d = 0$, but this number is
  $q^{n-2}$ since $n-1$ is odd \cite[Theorem 1.26]{wan}. 

  Thus, we have provided a $q$-to-$1$ map from a subset of $\Symz(n,
  k)$ into $\Symz(n, k+1)$ such that the complement of the
  domain has cardinality $q$ times that of the complement of the image. This proves the lemma.
\end{proof}


We use this lemma to prove the main result of this section.

\begin{proof}[Proof of Theorem~\ref{thm:clover}] 
  Let $n = 2m$.  We have 
\begin{align*}
  \symz(n,n) & = q^{-n} \symz(n,0) \\
  & = q^{-n} \sym(n) \\
  & =  q^{-2m} \prod_{i=1}^m \frac{q^{2i}}{q^{2i}-1} \prod_{i=0}^{2m-1} (q^{2m-i} -1) \\
  & = \prod_{i=1}^{m-1} \frac{q^{2i}}{q^{2i}-1} \prod_{i=0}^{2m-2}(q^{2m-1-i}-1) \\
  & = \sym(n-1),
\end{align*}
where in the first line we use Lemma~\ref{lemma:numberzeroes} and in
the rest we use Equation \eqref{eqn:invsymmetric}.
\end{proof}

\subsection{Skew-symmetric matrices} \label{sec:curious}

Let $\sk(n)$ denote the number of invertible $n \times n$
skew-symmetric matrices over $\F_q$.  It is not clear \emph{a priori} that there is any connection between these matrices and invertible symmetric matrices, but there is a curious relation between the two. In particular, we show that for $n$ even, the number of $n \times n$ invertible skew-symmetric matrices is the same as the number of $(n - 1) \times (n - 1)$ invertible symmetric matrices (and so, by Theorem~\ref{thm:clover}, also the same as the number of $n \times n$ invertible symmetric matrices with all-zero diagonal).  In addition, we explicitly count the invertible skew-symmetric matrices by rank, obtaining a $q$-analogue of involutions with a certain number of fixed points.

To begin, we prove a helpful technical lemma which states that the number of skew-symmetric matrices of a given rank is nearly independent of the first row of the matrix.

\begin{lemma} \label{lemma:skewfirstrow}
  The number of $n \times n$ skew-symmetric matrices of rank $r$ with first row $v$ is the same for all $v \neq (0, \ldots, 0)$.
\end{lemma}

\begin{proof}
  Let $v$ be any nonzero vector whose first entry is zero.  
  Permuting rows and columns symmetrically preserves skew-symmetry and rank, 
  so the number of skew-symmetric matrices over $\F_q$ of rank $r$ having $v$ 
  as first row is equal to the number of such matrices with any permutation 
  of $v$ as the first row.  
  Similarly, multiplying a row and the corresponding column by a scalar 
  preserves skew-symmetry and rank, so for any nonzero entry $v_i$ of $v$, 
  the number of skew-symmetric matrices of rank $r$ having $v$ as first row 
  is the same as the number of such matrices with $v_i$ replaced by any other 
  nonzero element of $\F_q$.  Thus, we only have to consider first rows of 
  the form $(0, 1, \ldots, 1, 0, \ldots, 0)$.  

We now give a bijection between skew-symmetric matrices of rank $r$ with first row $(0, 1^i, 0^{n - i - 1})$ and skew-symmetric matrices of rank $r$ with first row $(0, 1^{i + 1}, 0^{n - i - 2})$.  For a skew-symmetric matrix of rank $r$ with first row $(0, 1^i, 0^{n - i - 1})$, add row $i+1$ to row $i+2$ and then add column $i+1$
  to column $i+2$.  The resulting matrix is skew-symmetric of rank $r$ and has first row $(0, 1^{i + 1}, 0^{n - i - 2})$.  Moreover, this operation is reversible, so it is our desired bijection.  Together with the previous paragraph, this completes our result.
\end{proof}

\begin{proposition} \label{prop:curious} For $n$ even, $\sk(n) = \sym(n-1)$.
\end{proposition}

\begin{remark}
After the first write-up of this paper we found that this was proven by Jones \cite[Theorems 1.7, 1.7$'$, 1.8$'$, 1.9]{oj} using topological methods. 
\end{remark}

\begin{proof} We proceed by showing that the two sides of the claimed equality satisfy the same recursion.
  By Lemma~\ref{lemma:skewfirstrow}, the number of
  invertible skew-symmetric matrices of size $n$ is equidistributed
  with respect to the nonzero choices for the first row (and there 
  are no invertible matrices with first row all zero).  Thus to 
  compute $\sk(n)$ we multiply $q^{n - 1}-1$, the number of choices 
  for the first row, by the number of invertible skew-symmetric 
  matrices with first row $(0, 1, 0, \ldots, 0)$.  Let 
   \[
    A = \begin{bmatrix} 
    0      & 1        & 0       & \cdots & 0       \\ 
    -1     & 0        & a_{2,3} & \cdots & a_{2,n} \\ 
    0      & -a_{2,3} &         &        &         \\ 
    \vdots & \vdots   &         & B      &         \\ 
    0      & -a_{2,n} &         &        &
    \end{bmatrix}
   \] 
  be one such matrix, where $B$ is the lower-right $(n - 2)\times(n - 2)$ 
  block of $A$.  The matrix $B$ is certainly skew-symmetric; in addition, 
  $\det A = \det B$ and so $A$ is invertible if and only if $B$ is invertible.  Furthermore, the determinant of $A$ is independent of the $n - 2$ values $a_{2, 3}, \ldots, a_{2, n}$ that appear in the second row and column of $A$, so for each $B$ we can choose these values freely. Thus, we get the recurrence 
  \[ 
  \sk(n) = q^{n-2}(q^{n-1} - 1)\sk(n-2).
  \]
  For $n$ even, Equation \eqref{eqn:invsymmetric} implies 
  \[
  \sym(n - 1) = q^{n-2} (q^{n-1} - 1) \sym(n-3).
  \]
  Since $\sym(1) = \sk(2) = q-1$, we're done by induction.
\end{proof}

In fact, it is not difficult to enumerate skew-symmetric matrices of arbitrary rank over $\F_q$.  We do this in the following result.

\begin{proposition} Let $\sk(n,r)$ be the number of $n\times n$
  skew-symmetric matrices of rank $r$. When $r$ is odd we have $\sk(n, r) = 0$ and when $r$ is even we have
\[ \sk(n, r) = q^{r(r-2)/4}(1-q)^{r/2} \cdot \frac{[n]_q!}{[n - r]_q! \cdot [r]_q!!} .
\]
\end{proposition}
\begin{proof}
  It is a well-known fact that all skew-symmetric matrices have even
  rank (see for example \cite[\S XV.8]{lang}). Given an $n \times n$
  skew-symmetric matrix $A$ of rank $r$, we write $A = \begin{bmatrix}
    0 & v \\ -v^T & B\end{bmatrix}$ where $B$ is an $(n - 1) \times (n
  - 1)$ skew-symmetric matrix.  As we observed in the proof of
  Proposition~\ref{prop:Greta's recursion with k zeroes}, we have $\rk
  A - \rk B \in \{0, 1, 2\}$.  Since the rank of both matrices is even
  we have that $\rk B = r$ or $\rk B = r - 2$.  Thus, we have that
  $\rk B = r$ if and only if $v$ is in the rowspace of $B$.  It
  follows immediately that
  \[
    \sk(n,r) = q^r\sk(n-1,r) + (q^{n-1}-q^{r-2})\sk(n-1,r-2),
  \] 
  with initial values $\sk(n, 0) = 1$. One can easily verify by induction that the solution to this recursion is
  \[ 
    \sk(n, r) = q^{r(r-2)/4}(1-q)^{r/2} \cdot \frac{[n]_q!}{[n - r]_q! \cdot [r]_q!!}
  \]
  for $r$ even. 
\end{proof}

One interesting observation is that this is a $q$-analogue of $\binom{n}{r}(r-1)!!$, the number of ``partial involutions of rank $r$'' with no fixed points, i.e., the number of pairs of an $r$-subset of $\{1, \ldots, n\}$ together with a fixed point-free involution on that set.

\subsection{The curious relation via Schubert
  varieties} \label{sec:curious2} 

Let $n$ be even. In this section, we give another proof of
Proposition~\ref{prop:curious} via Schubert varieties. The idea is to
first realize both as the intersections of certain Schubert cells and
opposite Schubert cells. Such intersections are indexed by intervals
in certain Bruhat orders and the number of $\F_q$-valued points is
given by the (parabolic) R-polynomials of Deodhar. We then show that
the intervals in question are isomorphic as abstract posets and use
combinatorial invariance properties of these R-polynomials to get the
desired result. We will use \cite[Chapters 6 and 7]{smt} as a
reference for the following.

We let $I_m$ be the $m \times m$ identity matrix and $J_m$ be the $m
\times m$ matrix with $1$'s on the antidiagonal and $0$'s
elsewhere. Assume that $q$ is odd. (When $q$ is even,
Proposition~\ref{prop:curious} becomes Theorem~\ref{thm:clover}.) Let
$K$ be an algebraic closure of $\F_q$. Let $V$ be a
$2(n-1)$-dimensional vector space over $K$ with a nondegenerate
skew-symmetric bilinear form $\langle\, , \rangle$.  Let
\[
X = \{ V' \subset V \mid \dim V' = n-1, \text{ the restriction of the
  form to } V' \text{ is zero.} \}.
\]
be the Lagrangian Grassmannian. This is a closed subvariety of the
usual Grassmannian of $(n-1)$-dimensional subspaces of $V$.

Choose an ordered basis $\{e_1, \dots, e_{n-1}, e^*_{n-1}, \dots,
e^*_1\}$ of $V$ such that $\langle e_i, e^*_j \rangle =
\delta_{i,j}$. So we can think of elements of $X$ as $2(n-1) \times
(n-1)$ matrices where the columns of the matrix are some basis for the
given point. This also gives us a choice of Borel and opposite Borel
subgroup of the symplectic group $\mathbf{Sp}(V)$ acting on $X$, so we
can define Schubert varieties and opposite Schubert varieties. Then we
can embed symmetric $(n-1) \times (n-1)$ matrices into $X$ via the map
\[
M \mapsto \begin{bmatrix} I_{n-1} \\ J_{n-1}M \end{bmatrix}.
\]
The image of this map is the opposite big cell, and to get the set of
symmetric matrices with rank at most $r$ for some given $r$, we
intersect with an appropriate Schubert variety. So instead of counting
invertible symmetric matrices, we can count symmetric matrices with
rank at most $n-2$. We can index the Schubert varieties of $X$ by the
set
\[
\{ (a_1, \dots, a_{n-1}) \mid 1 \le a_1 < \cdots < a_{n-1} \le
2(n-1),\ \#(\{i,2n-1-i\} \cap \{a_1, \dots, a_{n-1}\}) = 1 \text{ for
  all } i \}.
\]
and the Bruhat order is given by termwise comparison, i.e., $(a_1,
\dots, a_{n-1}) \le (a'_1, \dots, a'_{n-1})$ if and only if $a_i \le
a'_i$. In particular, the Schubert variety we intersect with to get
symmetric matrices with rank at most $r$ is indexed by $(r+1, r+2,
\dots, n-1, 2n-1-r, 2n-r, \dots, 2n-2)$.

~

We can also do the same setup for skew-symmetric matrices. Let $W$ be
a $2n$-dimensional vector space over $K$ with a nondegenerate
symmetric bilinear form $(\, ,)$. Let
\[
Y' = \{ W' \subset W \mid \dim W' = n, \text{ the restriction of the
  form on } W' \text{ is zero.} \}
\]
be the orthogonal Grassmannian.

Choose an ordered basis $\{e_1, \dots, e_{n}, e^*_{n}, \dots, e^*_1\}$
of $W$ such that $(e_i, e^*_j) = \delta_{i,j}$. Then $Y'$ has two
connected components, each isomorphic to the spinor variety.  Let $Y$
be the component which contains the subspace spanned by $\{e_1, \dots,
e_{n-1}\}$. So we can think of elements of $Y$ as $2n \times n$
matrices where the columns of the matrix are some basis for the given
point. This also gives us a choice of Borel and opposite Borel
subgroup of the special orthogonal group $\mathbf{SO}(V)$ acting on
$Y$, so we can define Schubert varieties and opposite Schubert
varieties. Then we can embed skew-symmetric $n \times n$ matrices into
$Y$ via the map
\[
M \mapsto \begin{bmatrix} I_n \\ J_nM \end{bmatrix}.
\]
The image of this map is the opposite big cell, and to get the set of
skew-symmetric matrices with rank at most $r$ for some even integer
$r$, we intersect with an appropriate Schubert variety. So instead of
counting invertible skew-symmetric matrices, we can count
skew-symmetric matrices with rank at most $n-2$. We can index the
Schubert varieties of $Y$ by the set
\begin{align*}
  \left\{ (a_1, \dots, a_{n}) \;\middle\vert\; \begin{array}{l} 1
      \le a_1 < \cdots < a_{n} \le 2n,\\
      \#(\{i,2n+1-i\} \cap \{a_1, \dots, a_n\}) = 1 \text{
        for all } i, \\
      \#(\{a_1, \dots, a_n\} \cap \{n+1, \dots, 2n\}) \text{ is
        even } \end{array} \right\},
\end{align*}
and the Bruhat order is given by termwise comparison, i.e., $(a_1,
\dots, a_n) \le (a'_1, \dots, a'_n)$ if and only if $a_i \le
a'_i$. In particular, the Schubert variety we intersect with to get
skew-symmetric matrices with rank at most $r$ ($r$ even) is indexed by
$(r+1, r+2, \dots, n, 2n+1-r, 2n+2-r, \dots, 2n)$.

We also note that these two Bruhat orders are isomorphic in such a way
that the Schubert varieties corresponding to matrices of rank at most
$n-2$ in the two cases correspond to one another. First define
$\phi(i) = i$ if $ i \le n$ and $\phi(i) = i+2$ if $ i > n-1$. The map
from the first Bruhat order to the second is $(a_1, \dots, a_{n-1})
\mapsto (\phi(a_1), \dots, \phi(a_{n-1})) \cup \{x\}$ where $x \in
\{n,n+1\}$ is whichever element is needed to ensure that the result
satisfies the right evenness condition from above.

In both cases, we can count the number of points over $\F_q$ in the
intersection using the parabolic R-polynomials of Deodhar
\cite[Proposition 4.2]{deodhar}. These only depend on the
combinatorics of the Bruhat order since both the Lagrangian
Grassmannian and the spinor variety are cominuscule homogeneous spaces
(see for example \cite[Corollary 4.8]{brenti}). Hence the number of
matrices in both cases of rank at most $n-2$ are the same. Since $n$
is even, a non-invertible skew-symmetric $n \times n$ matrix has rank
at most $n-2$. It is clear that the number of skew-symmetric matrices
$n \times n$ matrices is the same as the number of symmetric $(n-1)
\times (n-1)$ matrices, so we also get that the number of invertible
such matrices are the same.

\section{Refined enumeration of symmetric
  matrices} \label{sec:symmetricrank}

In this section, we attack the problem of computing the number of $n \times n$ symmetric matrices over $\F_q$ with all-zero diagonal by rank.  Roughly speaking, we should expect this problem to be a $q$-analogue of counting fixed point-free involutions, or of ``partial fixed point-free involutions'' when we consider matrices of less than full rank.  As in the preceding sections, we construct a recursion to count the desired objects.  Our basic approach is the same as in Section \ref{sec:Greta's recursion}.  The main difference is that the symmetry of our matrices forces us to introduce a sort of parity condition depending on whether or not we can write a matrix in the form $M \cdot M^T$ for some other matrix $M$. The details on whether or not we can do this are different for odd and even characteristic. We begin by mentioning both cases to then restrict our attention and results to the odd case.

\begin{remark}[{\bf$q$ even}] It was shown by Albert \cite[Thm. 7]{albert} that a symmetric matrix $A$ in $\GL(n,q)$ can be factored in the form $A = M\cdot M^T$ for some matrix $M$ in $\GL(n,q)$ if and only if $A$ has at least one nonzero diagonal entry. Thus
\[
\#\{ A \in \Sym(n) \mid A = M\cdot M^T \text{ for some } M\in
\GL(n,q)\} = \sym(n) - \sym_0(n).
\] 
MacWilliams \cite{macwilliams} gave an elementary proof of Albert's theorem and also calculated $\sym_0(n,r)$, the number of $n \times n$ symmetric matrices of rank $r$ with zero diagonal, when $q$ is even.

\begin{theorem*}[{\cite[Thm. 3]{macwilliams}}] 
For $q$ even, if $r=2s+1$ is odd then
\begin{equation} \label{mac:qeveneq1} 
\sym_0(n,2s+1)=0
\end{equation}
while if $r=2s$ is even then
\begin{equation} \label{mac:qeveneq2}
\sym_0(n,2s)=\prod_{i=1}^{s} \frac{q^{2i-2}}{q^{2i}-1} \prod_{i=0}^{2s-1}(q^{n-i}-1).
\end{equation}
\end{theorem*}
Henceforth, we will always assume that $q$ is odd.
\end{remark}

For $q$ odd, define $\F_q^{*2}$
to be the set of perfect squares in $\F_q$, i.e., $x \in \F_q^{*2}$ if
and only if there is some $y \in \F_q$ such that $y^2 = x$.  Define
$\psi \colon \F_q^\times \to \{+, -\}$ by $\psi(\delta) = +$ if and only if
$\delta \in \F_q^{*2}$.  In other words, $\psi$ is the Legendre symbol for $\F_q$.  This notion can be extended to symmetric matrices in a natural way, as the following remark shows.

\begin{remark} \label{symm:symmdiagon} 
By applying symmetric row and column reductions, every $n \times n$ symmetric matrix $A$ of rank $r > 0$ can be written either in the form $A = M \cdot \diag(1^r, 0^{n - r}) \cdot M^T$ for some $M \in \GL(n, q)$ or in the form $M \cdot \diag(1^{r - 1}, z, 0^{n - r}) \cdot M^T$ for some $z \in \F_q \smallsetminus \F_q^{*2}$ and some $M \in \GL(n, q)$. 
\end{remark}

In the former case we say that $A$ has \textbf{(quadratic) character} $\psi(A) = \psi(1) = +$ and in the latter case we say it has character $\psi(A) = \psi(z) = -$. By convention $\psi({\bf 0})=+$ where ${\bf 0}$ is the zero matrix.  (This will be used in the proof of Proposition \ref{symm:recranknumzeroes}.)  Two notable special cases are that if $A \in \GL(n, q)$ then $\psi(A) = \psi(\det A)$, while if $A$ is diagonal then $\psi(A) = +$ if and only if the product of the nonzero diagonal entries of $A$ is a square in $\F_q$.

Let $\sq^{+}(m)$ be the number of solutions to the equation $\sum_{i = 1}^m x_i^2 = 0$ in $\F_q^m$, and similarly let $\sq^{-}(m)$ be the number of solutions to the equation $\sum_{i = 1}^{m - 1} x_i^2 + z x_m^2= 0$ where $z$ is some (fixed) non-square in $\F_q$.  Simple formulas are known for $\sq^{+}$ and $\sq^{-}$ (see \cite[Thm 1.26]{wan}) and we list them in Table \ref{table:sq}.
  
\begin{table}
\begin{center}$
\begin{array}{|l||l|l|l|}
\hline 
         & m \text{ odd} & m \text{ even} \\ 
\hline
\sq^{+}(m) & q^{m-1} & \begin{cases}
q^{m-1} + q^{m/2} - q^{m/2-1} &\text{if } (-1)^{m/2}\in \F_q^{*2},  \\ 
q^{m-1} -q^{m/2} + q^{m/2-1} &\text{otherwise }
\end{cases}\\ 
\hline
\sq^{-}(m) & q^{m-1} & \begin{cases}
q^{m-1} - q^{m/2} + q^{m/2-1}
&\text{if } (-1)^{m/2} \in \F_q^{*2},\\
q^{m-1} + q^{m/2} - q^{m/2-1}
& \text{otherwise }
\end{cases} \\ 
\hline
\end{array}
$\end{center}
\caption{Values of the functions $\sq^{+}$ and $\sq^{-}$, from \cite[Thm 1.26]{wan}.  Note that the case $\sq^{-}(m)$ for odd $m$ is not mentioned explicitly in the reference, but that the proof method for $m$ odd does not discriminate between the two cases $\sq^{\pm}(m)$.}
\label{table:sq}
\end{table}

Let $\sym(n,r)$ be the number of $n\times n$ symmetric matrices with rank $r$, and let $\sym^{\psi}(n,r)$ be the number of $n\times n$ symmetric matrices with rank $r$ and character $\psi$.  We will make substantial use the following results of MacWilliams.
\begin{theorem*}[{\cite[Thm. 2]{macwilliams}}] 
  We have
\begin{align}
   \sym(n,r) &= \prod_{i=1}^{\lfloor r/2\rfloor} \frac{q^{2i}}{q^{2i} - 1} \prod_{i=0}^{r-1} (q^{n-i}-1), \label{macwill:eq1} \\
   \sym^{+}(n,2s+1) &= \frac{1}{2}\sym(n,2s+1), \label{macwill:eq2} \\
 \intertext{and} 
   \sym^{+}(n,2s) &= \begin{cases}
     \displaystyle \frac{q^s + 1}{2q^s}\sym(n,2s) &\text{ if  $-1$ is a
      square  in } \F_q\\ 
     \displaystyle \frac{q^s + (-1)^s}{2q^s}\sym(n,2s) &\text{ otherwise}.
    \end{cases}
 \label{macwill:eq3}
\end{align}
\end{theorem*}

We now give several propositions that culminate in a recurrence for the number of symmetric matrices over $\F_q$ of rank $r$ with $k$ zeroes on the diagonal (Proposition \ref{symm:recranknumzeroes}). We use this recurrence to enumerate invertible symmetric matrices over $\F_q$ with zero diagonal (Corollary \ref{symm:enumsymminv}).  As before, we proceed by building larger matrices by adding rows and columns to smaller matrices; for each $n \times n$ symmetric matrix $B$ with zero diagonal, we consider the $(n + 1) \times (n + 1)$ matrices with zero diagonal of the form
\[
A = \begin{bmatrix} 0 & v \\ v^T & B\end{bmatrix}
\]
and analyze these matrices to write down a recursion.  First, as a warm-up, we consider matrices with all-zero diagonal by rank; the appearance of the functions $\sq^{\pm}$ indicate that a finer result is needed to reach our goal.  

\begin{proposition}\label{prop:first symmetric}
Let $B$ be a symmetric $n \times n$ matrix of rank $r\geq 1$ and quadratic character $\psi(\delta)$.  Of the $q^n$ symmetric matrices $A = \begin{bmatrix} 0 & v \\ v^T & B\end{bmatrix}$ we have 
\begin{compactenum}[\rm (i)]
\item $\sq^{\psi(\delta)}(r)$ matrices of rank $r$, 
\item $q^r - \sq^{\psi(\delta)}(r)$ matrices of rank $r + 1$, and
\item $q^n - q^r$ matrices of rank $r + 2$.
\end{compactenum}
\end{proposition}

\begin{proof} 
The proof proceeds along the same lines as Proposition \ref{prop:Greta's recursion with k zeroes}.  As noted in Remark \ref{symm:symmdiagon}, if $r = \rk(B) > 0$ and $\delta \in \F_q^*$ is such that $\psi(B) = \psi(\delta)$ then there exists $M \in \GL(n, q)$ such that $B = M \cdot \diag(1^{r - 1},\delta, 0) \cdot M^T$.  In this case, setting $D=\diag(1^{r-1},\delta,0)$, we have 
\[
\begin{bmatrix} 1 & 0 \\ 0 & M \end{bmatrix} \cdot  \begin{bmatrix} 0 & v \\ v^T & B \end{bmatrix} \cdot \begin{bmatrix} 1 & 0 \\ 0 & M^T \end{bmatrix} = \begin{bmatrix} 0 & v M^T \\ (vM^T)^T & D \end{bmatrix}.
\]
Since $M$ is invertible and we are interested in letting $v$ vary over all row vectors of length $n$, we may define $w = v M^T$ and let $w$ vary over all row vectors of length $n$, instead.
\begin{description}
\item[Case 1:] $\rk(A) = r + 2$.  As in the results of Section \ref{sec:Greta's recursion}, we have $\rk(A) = r + 2$ if and only if $w$ has a nonzero entry among its last $n - r$ entries (or equivalently, if $v \not \in \rowspace(B)$).  There are $q^n - q^r$ such choices for $w$, so $q^n - q^r$ matrices $A$ have rank $r + 2$.

\item[Case 2:] $\rk(A) = r$.  As in Case 1 of the proof of Proposition
  \ref{prop:Greta's recursion with k zeroes}, in order to have $\rk(A)
  = r$, we must have both that $w = (w_1, \ldots, w_r, 0, \ldots, 0)$
  and that using Gaussian elimination to kill the entries of $w$ and
  $w^T$ results in a matrix whose $(1, 1)$-entry is equal to $0$.
  This $(1, 1)$-entry is equal to $-(w_1^2 + \cdots + w_{r - 1}^2 +
  \delta^{-1} w_r^2)$, so there are $\sq^{\psi(\delta^{-1})}(r) =
  \sq^{\psi(\delta)}(r)$ choices of $w$ for which this entry is zero
  and thus $\sq^{\psi(\delta)}(r)$ matrices $A$ of rank $r$.

\item[Case 3:] $\rk(A) = r + 1$.  All $q^r - \sq^{\psi(\delta)}(r)$ choices of $w$ not yet accounted for result in a matrix $A$ of rank $r + 1$.
\end{description}

These three cases are exhaustive, so we have the claimed result.
\end{proof}

The appearance of the function $\sq^{\psi(B)}$ means that we cannot use Proposition \ref{prop:first symmetric} to write down a recursion.  Instead, we need a finer enumeration that also relates the character of the larger matrix $A$ to the character of $B$ and the choice of $v$.  The following result provides this recursion.  As usual, consider $B$ to be a fixed $n \times n$ symmetric matrix with all-zero diagonal.

\begin{proposition} \label{prop:second symmetric}
Let $B$ have rank $r\geq 1$ and quadratic character $\psi(\delta)$.  Of the $q^n$ total choices for $A = \begin{bmatrix} 0 & v \\ v^T & B \end{bmatrix}$, we have
\begin{compactenum}[\rm (i)]
\item $q^n-q^r$ matrices of rank $r+2$ and character $\psi(-\delta)$,

\item $\sq^{\psi(\delta)}(r)$ matrices of rank $r$ and character $\psi(\delta)$,

\item $\frac{1}{2}(\sq^{\psi(\delta)}(r+1) - \sq^{\psi(\delta)}(r))$ matrices of rank $r+1$ and character $\psi(\delta)$, and

\item the remaining $q^r - \sq^{\psi(\delta)}(r) - \frac{1}{2}(\sq^{\psi(\delta)}(r+1) - \sq^{\psi(\delta)}(r)) = q^r - \frac{1}{2} (\sq^{\psi(\delta)}(r + 1) + \sq^{\psi(\delta)}(r))$ matrices of rank $r+1$ and character $-\psi(\delta)$.
\end{compactenum}
\end{proposition}

\begin{proof}
As in the preceding results, write $D = \diag(1^{r - 1}, \delta, 0^{n - r})$ and choose $M \in \GL(n, q)$ such that $B = M \cdot D \cdot M^T$.  Then we wish to consider all $q^n$ matrices of the form 
\begin{equation}\label{eqn:diagonalize}
A = \begin{bmatrix} 1 & 0 \\ 0 & M\end{bmatrix} \cdot \begin{bmatrix} 0 & w \\ w^T & D\end{bmatrix} \cdot \begin{bmatrix} 1 & 0 \\ 0 & M^T \end{bmatrix}
\end{equation}
as $w$ varies in $\F_q^n$. 

\begin{description}
\item[Case 1:] $\rk(A) = r + 2$.  By applying further row and column reductions in Equation \eqref{eqn:diagonalize} we may write 
\[
A = R \cdot \begin{bmatrix} 
0      & 0 & \cdots & 0 & 1 \\ 
0      &   &        &   &   \\
\vdots &   &   D    &   &   \\ 
0      &   &        &   &   \\ 
1      &   &        &   & 
\end{bmatrix} \cdot R^{T}
\]
for some invertible matrix $R$, whence $\psi(A) = \psi\left(\begin{bmatrix} 0 & 0 & 1 \\ 0 & \delta & 0 \\ 1 & 0 & 0\end{bmatrix}\right) = \psi(-\delta)$.  Thus, all $q^n - q^r$ matrices $A$ of rank $r + 2$ have character $\psi(-\delta)$.

\item[Case 2:] $\rk(A) = r$.  In this case, there is some invertible $R$ such that $R \cdot A \cdot R^T = \begin{bmatrix} 0 & 0 \\ 0 & D\end{bmatrix}$, whence $\psi(A) = \psi(B) = \psi(\delta)$.  Thus, all $\sq^{\psi(\delta)}(r)$ matrices $A$ of rank $r$ have character $\psi(\delta)$.

\item[Case 3:] $\rk(A) = r + 1$.  In this case we have $w = (w_1,
  \ldots, w_r, 0, \ldots, 0)$.  Setting $a = -(w_1^2 + \cdots + w_{r -
    1}^2 + \delta^{-1} w_r^2) \neq 0$, there exists some invertible
  $R$ such that
\[
A = R \cdot \begin{bmatrix} a & 0 \\ 0 & D \end{bmatrix} \cdot R^T.
\]
If $a \in \F_q^{*2}$ then $\psi(A) = \psi(B) = \psi(\delta)$, and
otherwise $\psi(A) = - \psi(\delta)$.  We have $a \in \F_q^{*2}$ if
and only if there exists $x \in \F_q^*$ such that $x^2 + w_1^2 +
\cdots + w_{r - 1}^2 + \delta^{-1} w_r^2 = 0$, and the number of
choices of $w$ for which this equation has a solution is $\frac{1}{2}
\left(\sq^{\psi(\delta)}(r + 1) - \sq^{\psi(\delta)}(r)\right)$.
\end{description}
These cases are exhaustive, so we have our result.
\end{proof}

\begin{corollary}
We have the recursion
\begin{multline*}
\sym^{\psi}_0(n+1, r+1) = 
(q^n - q^{r-1})\sym_0^{\psi\cdot\psi(-1)}(n,r-1) + 
\sq^{\psi}(r + 1)\sym_0^{\psi}(n,r+1) + \\
\frac{1}{2}\left(\sq^{\psi}(r+1)  - \sq^{\psi}(r)\right) \sym_0^{\psi}(n,r) + \\
\left(q^{r} - \frac12(\sq^{-\psi}(r+1) + \sq^{-\psi}(r))\right)\sym^{-\psi}_0(n,r).
\end{multline*}
\end{corollary}

We are now ready to derive a recurrence for the number of symmetric matrices over $\F_q$ of rank $r$ with a prescribed number $k$ zeroes on the diagonal. We use the same approach as in Proposition~\ref{prop:Greta's recursion with k zeroes} and Lemma~\ref{lemma:numberzeroes}. 

\subsection{Recursions for fixed rank based on number of zeroes on the diagonal} \label{sec:4.1}

Let $\sym_0(n,k,r)$ be the number of $n \times n$ symmetric matrices with rank $r$ and the first $k$ diagonal elements equal to $0$, with no other restrictions. Thus, we have $\sym_0(n,n,r) = \sym_0(n,r)$ while $\sym_0(n,0,r) = \sym(n,r)$ (the value of which is given in Equation \eqref{macwill:eq1}).  Let $\sym_0^{\psi}(n,r,k)$ count only those matrices that have character $\psi$.

We define $\sym_0(r,\psi; n + 1, r',*)$  (respectively, $\sym_0^{\psi'}(r,\psi; n + 1, r', *)$) to be the number of $(n + 1) \times (n + 1)$ symmetric matrices $A$ of rank $r'$ (respectively, and character $\psi'$) associated to any matrix $B$ of rank $r$ and character $\psi$, and we define $\sym_0(r,\psi; n + 1, r', 0)$ (respectively, $\sym_0^{\psi'}(r,\psi; n + 1, r', 0)$)  to be the number of such matrices (respectively, of character $\psi'$) where in addition $a=A_{11}=0$.

\begin{proposition} \label{symm:recranknumzeroes} If $r$ is odd, define
  $t=0$ if $(-1)^{(r+1)/2}$ is a square in $\F_q$ and $t=1$
  otherwise. Then
\begin{multline*}
  \sym_0^{\psi}(n+1,k+1,r+1)  = \frac{1}{q} \sym_0^{\psi}(n+1,k,r+1) +  (-1)^t \cdot \psi\cdot\left(\frac{1}{2}\sym_0(n,k,r) + \sym_0^{\psi}(n,k,r+1)\right)\times \\
  \times (q^{(r+1)/2} - q^{(r-1)/2}).
\end{multline*}
If $r$ is even and $r > 0$, define $t=0$ if $(-1)^{r/2}$ is a
square in $\F_q$ and $t=1$ otherwise. Then
\[
\sym_0^{\psi}(n+1,k+1,r+1) = \frac{1}{q} \sym_0^{\psi}(n+1,k,r+1) -
\frac{(-1)^{t}}{2}(\sym_0^{+}(n,k,r)-\sym_0^{-}(n,k,r))(q^{r/2} -
q^{r/2-1}).
\]
We have initial values
\[
\sym_0^{\psi}(n+1,k+1,1) =  \frac{1}{2} \sym_0(n+1,k+1,1) = \frac{q - 1}{2}\sum_{i=0}^{n-k-1} q^{i} = \frac{q^{n-k} - 1}{2},
\]
\[
\sym_0^{+}(n,0,2s+1) = \frac{1}{2} \sym(n,2s+1),
\]
and 
\[
\sym_0^{+}(n,0,2s) = \frac{1}{2} \frac{q^s + (\psi(-1))^s}{q^s} \sym(n,2s).
\]
\end{proposition}

\begin{proof} 
As before, consider a symmetric $n \times n$ matrix $B$ of rank $r$ with $k$ prescribed zeroes on the diagonal, and choose $\delta \in \F_q^\times$ such that $\psi(B) = \psi(\delta)$; the matrix $A$ is obtained from $B$ by gluing on one row $v$ and one column $v^T$, and the rank of $A$ is then one of $r$, $r + 1$ and $r + 2$.  We split in three cases but most of the work has been done in Propositions \ref{prop:first symmetric} and \ref{prop:second symmetric}.  As in the preceding results, write $D = \diag(1^{r - 1}, \delta, 0^{n - r})$ and choose $M \in \GL(n, q)$ such that $B = M \cdot D \cdot M^T$.  Then we wish to consider all matrices of the form 
\begin{equation}\label{eqn:diagonalize:firstentry}
A = \begin{bmatrix} 1 & 0 \\ 0 & M\end{bmatrix} \cdot \begin{bmatrix} a & w \\ w^T & D\end{bmatrix} \cdot \begin{bmatrix} 1 & 0 \\ 0 & M^T \end{bmatrix}
\end{equation}
as $w$ varies in $\F_q^n$ and $a$ either varies over $\F_q$ or is fixed at $a = 0$. 

\begin{description}
\item[Case 1:] $\rk(A)=r+2$. As in the first case of the proof of Proposition \ref{prop:second symmetric}, we have $\psi(A)=\psi(-\delta)=\psi(-1)\cdot\psi(\delta)$ regardless of the value of the $(1,1)$-entry $a$ of $A$.  It follows immediately that
\begin{equation}\label{symm:rec1}
\sym_0^{\psi \cdot \psi(-1)}(r,\psi; n+1,r+2,0) = \frac{1}{q} \sym_0^{\psi \cdot \psi(-1)}(r,\psi; n+1,r+2,*),
\end{equation}
i.e., exactly a $q^{-1}$-fraction of all such matrices have
$(1,1)$-entry equal to $0$.  (This recursion holds when $r = 0$, i.e.,
when $B = {\bf 0}$ is the all-zero matrix, due to our convention that
$\psi({\bf 0}) = +$.)

\item[Case 2:] $\rk(A)=r$.  Much like in Case 2 of Proposition
  \ref{prop:first symmetric}, applying further symmetric row and
  column operations on the right-hand side of Equation
  \ref{eqn:diagonalize:firstentry} gives $A = R \cdot \begin{bmatrix}
    b & 0 \\ 0 & D\end{bmatrix} \cdot R^T$, where $b = a - (w_1^2 +
  \cdots + w_{r - 1}^2 + \delta^{-1} w_r^2)$.  Thus $\rk(A) = r$ if
  and only if $a - (w_1^2 + \cdots + w_{r - 1}^2 + \delta^{-1} w_r^2)
  = 0$, and in this case $\psi(A) = \psi(B) = \psi(\delta)$.  Now we
  consider two sub-cases:
\begin{compactenum}[(i)]
\item If $a$ may take any value in $\F_q$, then the value of $a$ is determined by the choice of $w$ and so $\sym_0^{\psi(\delta)}(r, \psi(\delta); n + 1, r, *) = q^r$.

\item If we restrict to $a = 0$ then by Case 2 of the proof of Proposition \ref{prop:second symmetric} we have that $\sym_0^{\psi(\delta)}(r,\psi(\delta); n + 1, r, 0) = \sq^{\psi(\delta)}(r)$.
\end{compactenum}
Using Table~\ref{table:sq}, the above two pieces of information imply
that
\begin{equation}\label{symm:rec2}
  \sym_0^{\psi}(r,\psi; n+1,r,0) = \begin{cases}
    \displaystyle \frac{1}{q}\sym_0^{\psi}(r,\psi; n+1,r,*),                           & r \text{ odd},\\
    \displaystyle \frac{1}{q}\sym_0^{\psi}(r,\psi; n+1,r,*) + (-1)^t
    \cdot \psi \cdot (q^{r/2} - q^{r/2-1}), & r \text{ even}. 
\end{cases}
\end{equation}
where $t=0$ if $(-1)^{r/2}$ is a square and $t=1$ otherwise.

\item[Case 3:] $\rk(A) = r + 1$.  Beginning as in Case 2, we must have
  $b = a - (w_1^2 + \cdots + w_{r - 1}^2 + \delta^{-1} w_r^2) \neq 0$.
  In this case, $\psi(A)=\psi(b)\psi(B)$.  We consider two sub-cases:
\begin{compactenum}[(i)]
\item If $a$ may take any value in $\F_q$ then for each choice of $w$ there are $\frac{q - 1}{2}$ choices of $a$ for which $\psi(b) = +$ and $\frac{q - 1}{2}$ choices of $a$ for which $\psi(b) = -$.  It follows that \\ $\sym_0^{\psi(\delta)}(r,\psi(\delta); n + 1, r+1, *) =\sym_0^{-\psi(\delta)}(r,\psi(\delta); n + 1, r+1, *)=\frac{1}{2}(q-1)q^{r}$.

\item If we restrict to $a = 0$ then we must count solutions to $b =
  -(w_1^2 + \cdots + w_{r - 1}^2 + \delta^{-1} w_r^2) \neq 0$
  depending on whether $b$ is a square.  When $b$ is a square (i.e.,
  when $\psi(A) = \psi(B)$) there are $\sq^{\psi(\delta)}(r + 1)$
  solution sets $(w_1, \ldots, w_r, \sqrt{b})$ to this equation.
  However, in $\sq^{\psi(\delta)}(r)$ of these we have $b = 0$, and
  the others are counted twice, once for each square root of $b$.
  Thus $\sym_0^{\psi(\delta)}(r,\psi(\delta); n + 1, r+1, 0) =
  \frac{1}{2} \left(\sq^{\psi(\delta)}(r + 1) -
    \sq^{\psi(\delta)}(r)\right)$.  All remaining matrices are counted
  by $\sym_0^{-\psi(\delta)}(r,\psi(\delta); n + 1, r+1, 0)$, so by
  subtraction we have $\sym_0^{-\psi(\delta)}(r,\psi(\delta); n + 1,
  r+1, 0) = q^r-\frac{1}{2} \left(\sq^{\psi(\delta)}(r + 1) +
    \sq^{\psi(\delta)}(r)\right)$.
\end{compactenum}

Using Table \ref{table:sq} and calculating carefully, we conclude that for $\epsilon \in \{\pm 1\}$ we have
\begin{multline}\label{symm:rec3}
\sym_0^{\epsilon \cdot \psi}(r,\psi; n+1,r+1,0) =\\
= \begin{cases} 
\frac{1}{q}\sym_0^{\epsilon \cdot \psi}(r,\psi; n+1,r+1,*) +  \frac{(-1)^{t_1}\cdot\epsilon \cdot \psi}{2}q^{(r-1)/2}(q-1),      & r \text{ odd},\\
\frac{1}{q}\sym_0^{\epsilon \cdot \psi}(r,\psi; n+1,r+1,*)  - \frac{(-1)^{t_2}\cdot\psi}{2}q^{r/2-1}(q-1),   & r \text{ even},
\end{cases}
\end{multline}
where $t_1=0$ if $(-1)^{(r+1)/2}$ is a square in $\F_q$ and $t_1 = 1$ otherwise, and $t_2=0$ if  $(-1)^{r/2}$ is a square in $\F_q$ and $t_2=1$ otherwise. 
\end{description}

As in the proof of Proposition \ref{prop:Greta's recursion with k zeroes}, we now change our perspective and consider the set of all $(n + 1)
\times (n + 1)$ symmetric matrices $A$ of rank $r + 1$ and character $\psi$ whose first $k + 1$
diagonal entries are equal to $0$.  Parametrizing these matrices by
the $n \times n$ submatrix $B$ that results from removing the first
row and first column, we have
\begin{equation} \label{symm:pfreceqA}
\sym_0^{\psi}(n + 1, k + 1, r + 1) = \sum_{B} \sym_0^{\psi}(\rk(B),\psi(B); n + 1, r + 1, 0)
\end{equation}
where the sum is over all $n \times n$ symmetric matrices $B$ whose first $k$ diagonal entries are zero.  The summands on the right are zero unless $\rk(B) \in \{r - 1, r, r + 1\}$, and so splitting the right-hand side according to the rank and character of $B$ gives
\begin{align} \label{symm:pfreceqB} 
\begin{split} 
  \sym_0^{\psi}(n + 1, k + 1, r + 1) =& \sym_0^{\psi \cdot
    \psi(-1)}(n, k, r - 1) \cdot
  \sym_0^{\psi}(r - 1,\psi \cdot \psi(-1); n + 1, r + 1, 0) \\
  &+ \sym_0^{\psi}(n, k, r + 1) \cdot \sym_0^{\psi}(r + 1,\psi; n + 1,
  r + 1, 0) \\
  &+ \sym_0^{\psi}(n,k, r) \cdot \sym_0^{\psi}(r,\psi; n + 1, r + 1, 0)\\
  & + \sym_0^{-\psi}(n,k, r) \cdot \sym_0^{\psi}(r,-\psi; n + 1, r +
  1, 0).
\end{split}
\end{align}

Now we may substitute from Equations \eqref{symm:rec1},
\eqref{symm:rec2}, and \eqref{symm:rec3} and collect the terms with
coefficient $\frac{1}{q}$ to get the desired result.

The initial values of the recursion are the case of rank one, $\sym_0^{\psi}(n+1,k+1,1)$, and the case when $k=0$, $\sym_0^{\psi}(n,0,r)$: 

\begin{description}
\item[Rank one:] Every such matrix $A$ has exactly one nonzero diagonal entry. So up to permuting rows and columns we assume the $(1,1)$ entry of $A$ is $\delta\neq 0$, in which case $A$ has the form $\begin{bmatrix} \delta & v \\ v^T & v^Tv/\delta\end{bmatrix}$.  Thus
\begin{equation}\label{eqn:diagonalizerank1}
A=\begin{bmatrix} \delta & 0 \\ v^T & I_{n}\end{bmatrix}  \cdot \begin{bmatrix} 1/\delta & 0 \\ 0 & {\bf 0_n}\end{bmatrix} \cdot \begin{bmatrix} \delta & v \\ 0 & I_n \end{bmatrix}
\end{equation}
and so $\psi(A)=\psi(\delta)$. From this we deduce that $\sym_0^{\psi}(n+1,k+1,1)=\frac{1}{2}\sym_0(n+1,k+1,1)$. To find $\sym_0(n+1,k+1,1)$ we do the following: since such matrices have rank $1$ and the first $k+1$ diagonal entries equal to zero, they are of the form 
$$
\begin{bmatrix} 
                0 & \hspace{-5pt} \cdots \hspace{-5pt} &  &  & 0  \\[-2pt]
            \vdots & \hspace{-5pt}\ddots \hspace{-5pt} &  & &\vdots   \\
                  &        & 0      & \hspace{-5pt}\cdots \hspace{-5pt} & 0  \\            
            &        & \vdots & B &   \\
0 & \hspace{-5pt}\cdots\hspace{-5pt} & 0 & &      
\end{bmatrix}
$$
where $B$ is a $(n-k)\times (n-k)$ rank one symmetric matrix with no
other restrictions. Hence
$$
\sym_0^{\psi}(n+1,k+1,1) = \frac{1}{2}\sym_0(n-k,0,1)=\frac{1}{2}\sym_0(n-k,1),
$$  
and the latter is given in \eqref{macwill:eq1}.

\item[$\mathbf{k=0}$:] We have $\sym_0^{\psi}(n,0,r)=\sym_0^{\psi}(n,r)$, the number of $n\times n$ symmetric rank $r$ matrices with no other restrictions. Depending on the parity of $r$ this is given in \eqref{macwill:eq2} or \eqref{macwill:eq3}.
\end{description}
This gives the desired result.
\end{proof}

We do not have a solution for this recurrence.  However, we use it to obtain two partial results towards its solution: 

\begin{corollary} \label{symm:partialsolrecrank}
We have
\[
\sym_0^+(n+1,,k+1,2s+1)=\sym_0^-(n+1,k+1,2s+1) = \frac{1}{2} \sym_0(n+1,k+1,2s+1),
\]
and
\[
\sym_0(n+1,k+1,2s) + \sym_0(n+1,k+1,2s+1) = \frac{1}{q^{k+1}}(\sym(n+1,2s) + \sym(n+1,2s+1)).
\]
\end{corollary}

\begin{proof}
From the case $r=2s$ in Proposition \ref{symm:recranknumzeroes} we have
\[
\sym_0^+(n+1,k+1,2s+1)-\sym_0^-(n+1,k+1,2s+1) = \frac{1}{q}(\sym_0^+(n+1,k,2s+1)-\sym_0^-(n+1,k,2s+1)).
\]
Since $\sym_0^+(n + 1, 0, 2s + 1) = \sym_0^-(n + 1, 0, 2s + 1)$, we have our result in this case.

Proposition \ref{symm:recranknumzeroes} also provides a recurrence for $\sym_0(n+1,k+1,r+1)=\sym_0^+(n+1,k+1,r+1) + \sym_0^-(n+1,k+1,r+1)$.  Setting $t=0$ if $(-1)^s$ is a square in $\F_q$ and $t=1$ otherwise, we have 
\begin{multline*}
\sym_0(n+1,k+1,2s+1) =
\frac{1}{q}\sym_0(n+1,k,2s+1) -\\
 - (-1)^t(\sym^+_0(n,k,2s)-\sym^-_0(n,k,2s))(q^{s}-q^{s-1})
\end{multline*}
and
\begin{multline*}
\sym_0(n+1,k+1,2s) =
\frac{1}{q}\sym_0(n+1,k,2s) + \\
 +(-1)^t(\sym^+_0(n,k,2s)-\sym^-_0(n,k,2s))(q^{s}-q^{s-1}).
\end{multline*}
Thus 
\[
\sym_0(n+1,k+1,2s) + \sym_0(n+1,k+1,2s+1) = \frac{1}{q}(\sym_0(n+1,k,2s) + \sym_0(n+1,k,2s+1)). 
\]
Iterating this, we obtain the desired formula.
\end{proof}

We also use Proposition \ref{symm:recranknumzeroes} to obtain an explicit formula in the case of invertible symmetric matrices (i.e., when $r = n$). We do this in the next section.

\subsection{Invertible symmetric matrices with a fixed number of zeroes on the diagonal}\label{sec:4.2}

Let $\symz(n,k)=\sym_0(n,k,n)$ be the number of invertible $n \times n$ symmetric matrices with first $k$ diagonal elements equal to $0$, with no other restrictions. Let $\symz^{\psi}(n,k)$ count only those matrices that have character $\psi$ (in this case the quadratic character of the determinant). We use the recursion in Proposition \ref{symm:recranknumzeroes} to give a recurrence for this full rank case.

\begin{proposition}\label{prop:symz recursions}
  If $n$ is odd define $t=0$ if $(-1)^{(n+1)/2}$ is a square in $\F_q$
  and $t=1$ otherwise. Then
\[
\symz^{\psi}(n+1,k+1) = \frac{1}{q} \symz^{\psi}(n+1,k) + \frac{(-1)^t\cdot \psi}{2} \symz(n,k)(q^{(n+1)/2} - q^{(n-1)/2}).
\]
If $n$ is even, define $t=0$ if $(-1)^{n/2}$ is a square in $\F_q$
and $t=1$ otherwise.  Then
\[
\symz^{\psi}(n+1,k+1) = \frac{1}{q} \symz^{\psi}(n+1,k) -
\frac{(-1)^{t}}{2} \left(\symz^{+}(n,k) -
  \symz^{-}(n,k)\right)(q^{n/2} - q^{(n-2)/2}).
\]

These recursions have initial values
\begin{equation}\label{symm:oddinitialvalues}
\symz^{+}(2m+1,0)= \frac{1}{2} \symz(2m+1,0)
\end{equation}
and 
\begin{equation} \label{symm:eveninitialvalues}
\symz^{+}(2m,0)= \frac{1}{2} \frac{q^m + (\psi(-1))^m}{q^m} \symz(2m,0).
\end{equation}
\end{proposition}

\begin{proof} 
Apply Proposition \ref{symm:recranknumzeroes} with $r = n$.
\end{proof}

Notice that in the first recursion (when $n$ is odd), the sign of the last summand depends on $\psi$. Exploiting this observation, we obtain simple recurrences for $\symz(m,k)=\symz^+(m,k) + \symz^-(m,k)$ (already given in Lemma \ref{lemma:numberzeroes}) and a more complicated one for $m$ odd.  We also give recurrences for $\symz^+(m,k)-\symz^-(m,k)$.

\begin{corollary} \label{symm:enumsymminv}
We have
\[
\symz(n+1,k+1) = \begin{cases}
  q^{-1} \symz(n+1,k) & \text{ if } n \text{ odd,}\\
  q^{-1} \symz(n+1,k) -
  (-1)^t(\symz^+(n,k)-\symz^-(n,k))(q^{n/2} - q^{(n-2)/2}) & \text{
    otherwise}
\end{cases}
\]
where $t=0$ if $(-1)^{n/2}$ is a square in $\F_q$ and
$t=1$ otherwise, and
\[
\symz^+(n+1,k+1)-\symz^-(n+1,k+1) = \begin{cases}
  \displaystyle  \frac{1}{q}(\symz^+(n+1,k)-\symz^-(n+1,k))  &
  \text{if } n \text{ odd,}\\ 
  \quad  + (-1)^{t}\cdot\symz(n,k)(q^{(n-1)/2}-q^{(n-3)/{2}})  & \\
  0 &\text{otherwise}
\end{cases}
\]
where $t=0$ if $(-1)^{(n+1)/2}$ is a square in
$\F_q$, and $t=1$ otherwise. We have initial values
$\symz(m,0)=\sym(2m)$ and
\[
\symz^+(m,0)-\symz^-(m,0) = \begin{cases}
\frac{\psi(-1)^m}{q^m} \sym(m) &\text{if } m \text{ is even},\\
0 & \text{otherwise.}
\end{cases}
\]
\end{corollary}

We know from Lemma \ref{lemma:numberzeroes} that for $n$ odd, we have $\symz(n+1,k+1)=\frac{1}{q^{k+1}}\sym(n+1)$.  We now provide the complementary result for $n$ even. 

\begin{theorem}\label{thm:clover2}
Let $\symz(n,k)$ be the number of invertible $n \times n$ symmetric matrices with the first $k$ diagonal elements equal to $0$ and let $\sym(n)$ be the number of invertible $n\times n$ symmetric matrices with no other restrictions.  We have
\[
\symz(2m,k+1)=\frac{1}{q^{k+1}}\sym(2m),
\]
and
\[
\symz(2m+1,k+1)=
\frac{q^{m^2+m}}{q^{k+1}} \sum_{j=0}^{\lfloor k/2 \rfloor+1} \! (-1)^j(q-1)^{m + j}[2m-2j+1]_q !!\left( \binom{k+1}{2j-1} + (q-1) \binom{k+1}{2j}\right).
\]
In terms of the character,
\begin{multline*}
  \symz^+(2m,k+1) = \frac{\sym(2m)}{2 q^{k+1}} + \\
    + \frac{(-1)^tq^{m^2}}{2
    q^{k+1}}\sum_{j=0}^{\lceil k/2\rceil} (-1)^j (q-1)^{m + j} [2m - 2j - 1]_q!! \left(\binom{k+1}{2j} + (q - 1)\binom{k + 1}{2j + 1}\right),
\end{multline*}

where $t=1$ if $(-1)^m$ is a square in $\F_q$ and $t=0$ otherwise;
and
\[
\symz^+(2m+1,k+1) = \frac{1}{2}\symz(2m+1,k+1).
\]
\end{theorem}

\begin{proof}
As in MacWilliams \cite[p 156]{macwilliams}, in the first attempt we iterate the recursions of Corollary \ref{symm:enumsymminv} to obtain the formulas above, but in hindsight we can directly verify that the formulas satisfy the recursions. 
\end{proof}

\section{Polynomiality, $q$-analogues, and some open questions} \label{sec:polynom} 
So far, we have fixed sets of the form $S = \{(i,i) \mid 1\leq i \leq k\}$, counted matrices over $\F_q$ with support avoiding $S$ by rank, and done analogous counts for symmetric and skew-symmetric matrices. In this section, we briefly examine what happens when we enumerate matrices of given rank whose support avoids an arbitrary fixed set of entries. 

\subsection{$q$-analogues and the proof of Proposition \ref{RL:qprop}}

Fix $m,n \geq 1$, $r
\geq 0$, and $S \subset \{(i,j) \mid 1 \le i \le m, 1 \le j \le
n\}$. Let $T_q=T_q(m\times n,S,r)$ be the set of $m \times n$ matrices $A$ over $\F_q$ with rank $r$ and support contained in the complement of
$S$.  We consider the problem of computing $\# T_q$, the number of such matrices.

A first observation is that, holding $m, n, r, S$ fixed and letting $q$ vary, the function $\#T_q$ need not be polynomial in $q$.  We have already seen this phenomenon in the case of symmetric matrices; for instance, setting $m = n = r$ to be an odd positive integer and $S=\{(i,i) \mid 1\leq i \leq n\}$ we have from Equations \eqref{mac:qeveneq1} and \eqref{mac:qeveneq2} and Theorem \ref{thm:clover2} that $\#(T_q(n\times n,S,n) \cap \Sym(n)) = \sym_0(n)$ is equal to zero when $q$ is even but is nonzero when $q$ is odd.  This lack of polynomiality also occurs in the not-necessarily symmetric case. Stembridge \cite[Section 7]{stem} showed that for $n=m=7$, if $S'$ is the complement of the incidence matrix of the Fano plane, then the number of invertible $7\times 7$ matrices in $\F_q$ whose support avoids $S'$ is given by two different polynomials depending on whether $q$ is even or odd.  (This is the smallest such example in the sense that $\#T_q(n\times n,S,n)$ is a polynomial if $n < 7$ for any set $S$, and if $n=7$ and $\#S > 28$.) See Figure \ref{fig:Fano} for a construction of $S'$. 

\begin{figure}\begin{center}
\raisebox{.75in}{$
\left[\begin{array}{ccccccc}
a_{{1\textcolor{red}{1}}}&a_{{12}}&0&0&0&0&a_{{17}}
\\ a_{{2\textcolor{red}{1}}}&0&a_{{23}}&0&0&a_{{26}}&0
\\ a_{{3\textcolor{red}{1}}}&0&0&a_{{34}}&a_{{35}}&0&0
\\ 0&a_{{42}}&a_{{43}}&0&a_{{45}}&0&0
\\ 0&a_{{52}}&0&a_{{54}}&0&a_{{56}}&0
\\ 0&0&a_{{63}}&a_{{64}}&0&0&a_{{67}}
\\ 0&0&0&0&a_{{75}}&a_{{76}}&a_{{77}}\end{array}\right]
$}
\hspace{1in}
\begin{tikzpicture}[scale=0.75,thick]
                \node [style=bv] (1) at (0, 3.464) [label=above:$1$] {};
                \node [style=bv] (3) at (1, 1.732) [label=above:$4$] {}; 
                \node [style=bv] (4) at (0, 1.155) [label=left:$7$] {};
                \node [style=bv] (5) at (-2, 0) [label=below:$3$] {};
                \node [style=bv] (6) at (0, 0) [label=below:$6$] {};
                \node [style=bv] (7) at (2, 0) [label=below:$5$]  {};
                \draw (0,1.155) circle (1.155cm);
                \draw[red] [-] (1) -- (5);
                \node [style=bv] (2) at (-1, 1.732) [label=above:$2$] {};
                \draw [-] (1) -- (7); 
                \draw [-] (1) -- (6);
                \draw [-] (7) -- (5);
                \draw [-] (5) -- (3);
                \draw [-] (7) -- (2);
\end{tikzpicture}
\end{center}
\caption{A representative matrix from $T_q(7\times 7,S',7)$ where $S'$ is the complement of the incidence matrix of the Fano plane, shown at right. Stembridge \cite{stem} showed this to be the smallest example where $\#T_q$ is not a polynomial in $q$.}
\label{fig:Fano}
\end{figure}
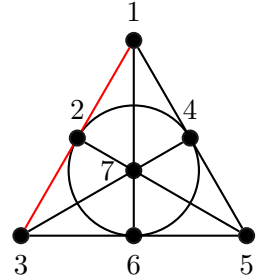

A second observation is that we expect $\# T_q$ to be a $q$-analogue of a closely related problem for permutations. Specifically, let $T_1=T_1(m\times  n,S,r)$ be the set of $0$-$1$ matrices with exactly $r$ $1$'s, no two of which lie in the same row or column, and with support contained in the complement of $S$. The following proposition makes this precise.

\begin{proposition} \label{RL:qprop}
Fix $m,n \geq 1$, $r
\geq 0$, and $S \subset \{(i,j) \mid 1 \le i \le m, 1 \le j \le
n\}$. Let $T_q=T_q(m\times n,S,r)$ be the set of $m \times n$ matrices $A$ over $\F_q$ with rank $r$ and support contained in the complement of
$S$, and $T_1$ be the set of $0$-$1$ matrices with exactly $r$ $1$'s, no two of which lie in the same row or column, and with support
contained in the complement of $S$. Then we have 
\[
\# T_q \equiv \# T_1 \cdot (q-1)^r \pmod {(q-1)^{r+1}}.
\]
\end{proposition}

In particular, for any infinite set of values of $q$ for which $\#T_q$ is a polynomial in $q$ we have that $(q-1)^r$ divides $\#T_q$ as a polynomial and that $\#T_q/(q-1)^r\mid_{q=1} = \#T_1$. 

\begin{proof}
  For each $\ell$, identify $(\F_q^\times)^\ell$ with the group of invertible
  diagonal $\ell \times \ell$ matrices. Consider the action of $(\F_q^\times)^m \times (\F_q^\times)^n$ on $T_q$ given by
  $(X,Y)\cdot A = XAY^{-1}$. For any $A \in T_q$, let $G$ be the bipartite
  graph with vertices $v_1, \dots, v_m, w_1, \dots, w_n$ and an edge
  $v_iw_j$ if $A_{ij} \neq 0$. Then $(x_1, \dots, x_m, y_1, \dots,
  y_n) \in (\F_q^\times)^m \times (\F_q^\times)^n$
  stabilizes $A$ if and only if $x_i=y_j$ for all edges $v_iw_j$ of
  $G$. Thus, the size of the stabilizer of $A$ is $(q-1)^{C(G)}$,
  where $C(G)$ is the number of connected components of $G$, and the
  size of the orbit of $A$ is therefore $(q-1)^{m+n-C(G)}$.

  Since $A$ has rank $r$, at least $r$ of the $v_i$ and $r$ of the
  $w_i$ have positive degree. It follows that $C(G) \leq m+n-r$ 
  with equality if and only if $G$ consists of $r$
  disjoint edges, that is, when $G$ is the graph associated to a matrix in
  $T_1$. It follows that the size of each orbit is $(q-1)^a$ for some
  $a \geq r$, and the number of orbits of size $(q-1)^r$ is $\#T_1$. 
\end{proof}

\begin{remark} 
  The technique in the proof of Proposition \ref{RL:qprop} is widely applicable to similar problems.  We give one brief example in the case of symmetric matrices.  Suppose that $q$ is odd.  The group $(\F_q^\times)^n$ of invertible diagonal matrices acts on the set of symmetric 
  matrices by the rule $X \cdot A = XAX$. For a symmetric matrix $A$, we 
  consider the graph $G$ on $n$ vertices $v_1,\ldots,v_n$ with edge $v_iv_j$ if 
  and only if $A_{ij} \neq 0$.  The order of the stabilizer of $A$ is the number 
  of tuples $(x_1, \ldots, x_n) \in (\F_q^\times)^n$ such that $x_ix_j=1$ whenever 
  $v_iv_j$ is an edge in $G$.  For each connected component of $G$ we have $q - 1$ 
  solutions if the component is bipartite or $2$ solutions if the component 
  contains odd cycles (including possibly loops).  Thus, if $C_{\text{bip}}(G)$ is the number of bipartite components of $G$ then the size of the stabilizer of $A$ is
  $(q-1)^{C_{\text{bip}}(G)} \cdot 2^{C(G) - C_{\text{bip}}(G)}$ and so the size of the orbit of $A$ is $(q-1)^{n - C(G)} \cdot ((q-1)/2)^{C(G) - C_{\text{bip}}(G)}$.  

  Now restrict consideration to matrices of rank $2s$ with zero diagonal.   In this case we have $C(G) \leq n - s$, with equality exactly when $G$ consists of $s$ disjoint edges.  The contribution of the orbits of such matrices is $(q - 1)^s$ times the number of symmetric $0$-$1$ matrices of rank $2s$ with no two ones in the same row or column, so (looking modulo $(q - 1)^{s + 1}$) we have that symmetric matrices with zero diagonal are a $q$-analogue of ``partial fixed point-free involutions.''
\end{remark}

\subsection{Polynomiality and a conjecture of Kontsevich}

As mentioned in Section \ref{sec:intro}, the question of the polynomiality of $\#T_q$ is related to the Kontsevich conjecture. We briefly provide some background on this conjecture and on its relation to the polynomiality of $\#T_q$. 

Let $G$ be an undirected connected graph with edge set $E$, and
form the polynomial ring $\ZZ[x_e \mid e \in E]$. We consider the
polynomial
\[
P_G(x) = \sum_T \prod_{e \notin T} x_e,
\]
where the sum is over all spanning trees $T$ of $G$. Motivated
by computer calculations and some relations to algebraic geometry,
Kontsevich speculated that the number of solutions to $P_G(x) \ne 0$
over $\F_q$ is a polynomial function in the parameter $q$. Stanley
\cite{spt} reformulated this as follows. First, consider the
renormalization
\[
Q_G(x) = P_G(1/x) \prod_{e \in E} x_e = \sum_T \prod_{e \in T} x_e.
\]
Let $g_G(q) = \#\{ x \in \F_q^E \mid Q_G(x) \ne 0\}$. Using
inclusion-exclusion, one finds that the number of solutions to $P_G(x)
\ne 0$ is a polynomial (in $q$) if and only if $g_G(q)$ is a polynomial. Let
$v_1, \dots, v_n$ be the vertices of $G$ and suppose that $v_n$ is
adjacent to all other vertices. By the matrix-tree theorem, one may
conclude that $g_G(q)$ is the number of symmetric matrices in
$\GL(n-1,q)$ such that the $(i,j)$-th entry is $0$ whenever $i \ne j$ and
$v_i$ and $v_j$ are not connected. Thus, $g_G(q)=\#( T_q(n\times n,S_G,n)\cap \Sym(n))$ where $S_G=\{ (i,j) \mid i\neq j \textrm{ and } v_iv_j\not\in E\}$. 

Using Stanley's reformulation, Belkale and Brosnan showed in \cite{bb}
that Kontsevich's speculation is false by showing that the functions
$g_G(q)$ are as complicated (in a very precise sense) as the functions
counting the number of solutions over $\F_q$ of any variety defined
over $\ZZ$. In addition, Stembridge showed in \cite{stem} that $g_G(q)$ is a polynomial for graphs $G$ with at most $12$ edges; in \cite{schn}, Schnetz extended this result to $13$ edges and found six non-isomorphic graphs with $14$ edges such that $g_G(q)$ is not a polynomial in $q$. Given these results, it becomes an interesting problem to
determine when $g_G(q)$ is a polynomial in $q$.  Taken together with Proposition~\ref{RL:qprop}, they also 
suggest the following question:

\begin{question}
  For which families of sets $S$ is $\# T_q(m\times n,S,r)$ a
  polynomial in $q$?
\end{question}

Note that  $\#T_q(m\times n,S,r)$ is invariant under permutations of rows and columns.
Below, we describe one class of sets $S$ for which the answer is already known by the theory of $q$-rook numbers.

Let $\overline{S}$ denote the complement of the set $S$.  We say that
$S\subseteq [n]\times [n]$ is a {\bf straight shape} if its elements
form a Young diagram.  Thus, to every integer partition $\lambda$
with at most $n$ parts and with largest part at most $n$ (i.e., to each sequence of integers
$(\lambda_1,\lambda_2,\ldots,\lambda_n)$ such that $n \geq \lambda_1\geq
\lambda_2 \geq \cdots \geq \lambda_n \geq 0$) there is an associated set $S = S_\lambda$. 
We have that $\#S_{\lambda}=\sum\lambda_i=|\lambda|$ is the sum of the parts of $\lambda$.  Similarly,
if $\lambda$ and $\mu$ are partitions such that $S_{\mu}\subseteq
S_{\lambda}$ then we say that the set $S_{\lambda}\backslash S_{\mu}$ has {\bf
  skew shape} and we denote it by $S_{\lambda/\mu}$. Figure \ref{fig:strtskshapes} gives examples of matrices in $T_q(n\times n,S,r)$ when $S$ is a straight shape, a skew shape, and the complement of a skew shape. Next we give three easy
facts about straight and skew shapes.

\begin{figure}
$$
\begin{array}{ccc}
S_{(4,3,2)} & S_{(5,5,4,3,1)/(2,2,1)} &
\overline{S_{(5,5,4,3,1)/(2,2,1)}}\\[2mm]
\left[ \begin {array}{ccccc} 0&0&0&0&a_{{15}}\\ 0&0
&0&a_{{24}}&a_{{25}}\\ 0&0&a_{{33}}&a_{{34}}&a_{
{35}}\\ a_{{41}}&a_{{42}}&a_{{43}}&a_{{44}}&a_{
{45}}\\ a_{{51}}&a_{{52}}&a_{{53}}&a_{{54}}&a_{
{55}}\end {array} \right],
&
\left[ \begin {array}{ccccc} 
a_{{11}}&a_{{12}}&0&0&0\\ 
a_{{21}}&a_{{22}}&0&0&0\\ 
a_{{31}}&0&0&0&a_{{35}}\\ 

0&0&0&a_{{44}}&a_{{45}}\\ 
0&a_{{52}}&a_{{53}}&a_{{54}}&a_{{55}}
\end {array} \right], 
&
\left[ \begin {array}{ccccc} 
0&0&a_{{13}}&a_{{14}}&a_{{15}}\\ 
0&0&a_{{23}}&a_{{24}}&a_{{25}}\\ 
0&a_{{32}}&a_{{33}}&a_{{34}}&0\\ 
a_{{41}}&a_{{42}}&a_{{43}}&0&0\\ 
a_{{51}}&0&0&0&0
\end {array} \right] 
\end{array}
$$
\caption{Representative matrices from $T_q(5\times5,S,r)$ when $S$ is a straight shape, a skew shape, and the complement of a skew shape.}
\label{fig:strtskshapes}
\end{figure}

\begin{remark} \label{pol:easyfactsshapes}
\begin{compactenum}[\rm (i)] 
\item Up to a rotation of $[n]\times [n]$, the complement $\overline{S_{\lambda}}$ of the straight shape $S_\lambda$ is also a straight shape.  However, $\overline{S_{\lambda/\mu}}$ is typically not a skew shape.
\item If $(i,j) \in S_{\lambda}$ then the rectangle
  $\{(s,t) \mid 1\leq s\leq i, 1\leq t\leq j\}$ is contained in
  $S_{\lambda}$. General skew shapes $S_{\lambda/\mu}$ do not have this
  property.
\item If $\lambda = (n,n-1,\ldots,2,1)$ and $\mu = (n-1,n-2,\ldots,1,0)$ are so-called ``staircase shapes'' then $S_{\lambda/\mu}$ is, up to rotation, the set of diagonal entries.  Thus the value $\#T_q(n\times n,S_{\lambda/\mu},n)$ is given in Proposition \ref{RL:prop2} while trivially $\#T_q(n\times n, \overline{S_{\lambda/\mu}},n)=\#\{\text{invertible diagonal matrices}\}=(q-1)^n$.
\end{compactenum}
\end{remark}

Given a set $S\subseteq [n]\times [n]$, the {\bf $r$ $q$-rook number} of
Garsia and Remmel \cite{garrem} is $R_r(S,q) = \sum_{C} q^{\inv(C,S)}$, where the sum is over all rook
placements $C\in T_1(n\times n,\overline{S},r)$ of $r$ non-attacking rooks in $S$ and where $\inv(C,S)$ is the number of squares in $S$ not directly above (in the same column) or to the left (in the same row) of any placed rook.

The following result of Haglund shows that when $S=S_{\lambda}$, we have that
$\#T_q(n\times n,S_{\lambda},n)$ is a polynomial, and in fact is the product of
a power of $q-1$ and a polynomial with nonnegative coefficients. We
reproduce Haglund's proof to emphasize where we use that $S$ is a
straight shape.
\begin{theorem*}[{\cite[Theorem 1]{jh}}]
For straight shapes $S_{\lambda}$,
\[
\#T_q(n\times n,S_{\lambda},r) = (q-1)^r q^{n^2-|\lambda|-r}R_r(\overline{S_{\lambda}},q^{-1}),
\]
\end{theorem*}

\begin{proof}[Proof sketch.] From Remark \ref{pol:easyfactsshapes} (i)
  it is equivalent to work with $T_q(n\times n,
  \overline{S_{\lambda}},r)$.  Choose a matrix
  $A$ in $T_q(n\times n, \overline{S_{\lambda}},r)$, that is, whose support 
  is in $S_{\lambda}$, and perform Gaussian elimination in the following
  order: traverse columns from bottom to top, starting with the last column.  
  When you come to a nonzero entry (i.e., a pivot), use it to eliminate the 
  entries above it in the same column and to its left in the same row.  
  Then move on to the next column and repeat.  The crucial point
  is that, by Remark \ref{pol:easyfactsshapes} (ii), each step in the
  elimination gives another matrix contained in $T_q(n\times n,
  \overline{S_{\lambda}},r)$. After elimination, the positions of the
  pivots are a placement of $r$ non-attacking rooks on $S_{\lambda}$, and the
  number of matrices in $T_q(n\times n, S_{\lambda},r)$ that give a
  fixed rook placement is $(q-1)^r q^{\#S_{\lambda}-r-\inv(C,S_{\lambda})}$.
\end{proof}

\begin{remark}
Haglund's theorem also implies that $\#T_q(n\times n,S,r)$ is a polynomial in $q$ for any set $S$ that can be {\em arranged} into a straight shape by permuting rows and columns since $\#T_q$ is invariant under these permutations.
\end{remark}

\begin{question}
  The proof above fails for $\overline{S_{\lambda/\mu}}$ by Remark
  \ref{pol:easyfactsshapes}(ii). However, computations using Stembridge's
  Maple package {\tt reduce} \cite{stemr} suggest that when $S$ is a skew shape, $\#T_q$ is still a 
  polynomial and that when $S$ is the complement of a skew shape, $\#T_q$ is a
  power of $q - 1$ times a polynomial with nonnegative coefficients.  Is this true for all
  skew shapes and their complements?    
\end{question}

(Recall that any counter-examples satisfy $n=7$ and $\#S\geq 28$ or $n \geq 8$.)

\begin{question}
  \noindent Haglund's theorem and the preceding question suggest
  similarities between $\#T_q$ for $S$ and $\overline{S}$ that is reminiscent of the classical reciprocity of rook placements and rook numbers (see 
  \cite{chow} for a short combinatorial proof). Dworkin \cite[Theorem
  8.21]{dwork} gave an analogue of this classical reciprocity for $q$-rook
  numbers $R_r(S,q)$ when $S=S_{\lambda}$. By Haglund's result, this implies a
  reciprocity formula relating $T_q(n\times n,S_{\lambda},r)$ and
  $T_q(n\times n,\overline{S_{\lambda}},r)$. Can this reciprocity be
  extended to skew or other shapes? If so, we could recover the
  formula for $f_{n,n}$ in Proposition \ref{RL:prop2} from the formula
  of its complement: $(q-1)^n$.
\end{question}

\small

\noindent Joel Brewster Lewis, Alejandro H. Morales,
Steven V Sam, and Yan X Zhang \\
Department of Mathematics, Massachusetts Institute of Technology \\
Cambridge, MA USA 02139\\
\{{\tt jblewis, ahmorales, ssam, yanzhang}\}{\tt @math.mit.edu}

\bigskip

\noindent Ricky Ini Liu \\
Department of Mathematics, University of Michigan\\
Ann Arbor, MI USA 48109\\
{\tt riliu@umich.edu}

\bigskip

\noindent Greta Panova \\
Department of Mathematics, UCLA\\
Los Angeles, CA USA 90095\\
{\tt panova@math.ucla.edu}

\end{document}